\documentclass[12pt]{amsart}

\usepackage[hmargin=0.8in,height=8.6in]{geometry}
\usepackage{amssymb,amsthm, times}
\usepackage{delarray,verbatim}
\usepackage{ifpdf}
\ifpdf
\usepackage[pdftex]{graphicx}
 \else
\usepackage[dvips]{graphicx}
 \fi

\usepackage{bm}

\usepackage{ifpdf}
\usepackage{color}
\definecolor{webgreen}{rgb}{0,.5,0}
\definecolor{webbrown}{rgb}{.8,0,0}
\definecolor{emphcolor}{rgb}{0.95,0.95,0.95}

\usepackage{hyperref}
\hypersetup{%
          colorlinks=true,
          linkcolor=webbrown,
          filecolor=webbrown,
          citecolor=webgreen,
          breaklinks=true}
\ifpdf \hypersetup{pdftex,
            pdfstartview=FitH, 
            bookmarksopen=true,
            bookmarksnumbered=true
} \else \hypersetup{dvips} \fi

\newcommand{\lapinv}{\Phi(q)}

\numberwithin{equation}{section}

\newtheorem{theorem}{Theorem}[section]
\newtheorem{proposition}{Proposition}[section]
\newtheorem{corollary}{Corollary}[section]
\newtheorem{remark}{Remark}[section]
\newtheorem{lemma}{Lemma}[section]

\newtheorem{assump}{Assumption}[section]

\numberwithin{remark}{section} \numberwithin{proposition}{section}
\numberwithin{corollary}{section}
\newcommand {\R}{\mathbb{R}}

\newcommand {\p}{\mathbb{P}}

\newcommand {\E}{\mathbb{E}}

\newcommand{\diff}{{\rm d}}

\newcommand{\lev}{L\'{e}vy }

\title[Optimal dividends in the dual model under transaction costs]{Optimal dividends in the dual model under transaction costs}
\thanks{This version: \today. }
\author[E. Bayraktar]{Erhan Bayraktar$^*$}
\thanks{$*$\,Department of Mathematics,
University of Michigan,
530 Church Street,
Ann Arbor, MI  48109-1043, USA. Email: erhan@umich.edu}
\author[A. E. Kyprianou]{Andreas E. Kyprianou$^\dag$  }
\thanks{$\dag$\,Department of Mathematical Sciences, 
The University of Bath, 
Claverton Down, 
Bath BA2 7AY, 
UK. Email: a.kyprianou@bath.ac.uk}
\author[K. Yamazaki]{Kazutoshi Yamazaki$\ddag$}
\thanks{$\ddag$\, (corresponding author) Department of Mathematics,
Faculty of Engineering Science, Kansai University, 3-3-35 Yamate-cho, Suita-shi, Osaka 564-8680, Japan. Email: \mbox{{\em
kyamazak@kansai-u.ac.jp}}.  Tel: +81-6-6368-1527. }
\date{}
\begin{document}

\begin{abstract}
We analyze the optimal dividend payment problem in the dual model under constant transaction costs. 
We show, for a general spectrally positive \lev process, an optimal strategy is given by a $(c_1,c_2)$-policy that brings the surplus process down to $c_1$ whenever it reaches or exceeds $c_2$ for some $0 \leq c_1 < c_2$.  The value function is succinctly expressed in terms of the scale function.  A series of numerical examples are provided to confirm the analytical results and to demonstrate the convergence to the no-transaction cost case, which was recently solved by Bayraktar et al.\ \cite{Bayraktar_2012}.
\\
\noindent \small{\textbf{Key words:} dual model; dividends; impulse control;
 spectrally positive \lev processes; scale functions.\\
\noindent JEL Classification: C44, C61, G24, G32, G35 \\
\noindent  AMS 2010 Subject Classifications: 60G51, 93E20}\\
\end{abstract}

\maketitle

\section{Introduction}

We solve the optimal dividend problem under fixed transaction costs in the so-called \emph{dual model}, in which the surplus of a company is driven by a \lev process with positive jumps  (\emph{spectrally positive L\'{e}vy} process).  This is an appropriate model for a company driven by inventions or discoveries.  The case without transaction costs has recently been well-studied; see \cite{MR2324568}, \cite{MR2372996}, \cite{Avanzi_2008}, and \cite{Avanzi_2011}.  In particular, in \cite{Bayraktar_2012}, we show the optimality of a barrier strategy  (reflected \lev process) for a general spectrally positive L\'{e}vy process of bounded or unbounded variation.  

A strategy is assumed to be in the form of impulse control; whenever dividends are accrued, a constant transaction cost $\beta > 0$ is incurred.   As opposed to the barrier strategy that is typically optimal for the no-transaction cost case, we shall pursue the optimality of the so-called $(c_1,c_2)$-policy that brings the surplus process down to $c_1$ whenever it reaches or exceeds $c_2$ for some $0 \leq c_1 < c_2 < \infty$.  While, as in \cite{Loeffen_2009_2, Thonhauser_2011}, an optimal strategy may not lie in the set of $(c_1,c_2)$-policies for the spectrally negative \lev case, we shall show that  it is indeed so in the dual model for any choice of underlying spectrally positive \lev process.  As a related work, we refer the reader to a compound Poisson dual model by \cite{Yao_2011} where transaction costs are incurred for capital injections.  In inventory control, the optimality of similar policies, called $(s,S)$-policies,  is shown to be optimal in \cite{Bensoussan_2009,Bensoussan_2005} for a mixture of a Brownian motion and a  compound Poisson process and in \cite{Yamazaki_2013} for a general spectrally negative \lev process.

Following \cite{Bayraktar_2012}, we take advantage of the fluctuation theory for the spectrally positive \lev process (see e.g.\ \cite{Bertoin_1996}, \cite{Doney_2007} and \cite{Kyprianou_2006}).  The expected net present value (NPV) of dividends (minus transaction costs) under a $(c_1,c_2)$-policy until ruin is first written in terms of the scale function.  We then show the existence of the maximizers $0 \leq c_1^* < c_2^* < \infty$ that satisfy the continuous fit (resp.\ smooth fit) at $c_2^*$ when the surplus process is of bounded (resp.\ unbounded) variation and that the derivative at $c_1^*$ is one when $c_1^* > 0$ and is less than or equal to one when $c_1^* = 0$.  These properties are used to verify the optimality of the $(c_1^*,c_2^*)$-policy. 

In order to evaluate the analytical results and to examine the connection with the no-transaction cost case developed by \cite{Bayraktar_2012}, we conduct a series of numerical experiments using \lev processes with positive i.i.d.\ phase-type jumps with or without Brownian motion \cite{Asmussen_2004}.  We shall confirm the existence of the  maximizers $0 \leq c_1^* < c_2^* < \infty$ and examine the shape of the value function at $c_1^*$ and $c_2^*$.  We further compute for a sequence of  unit transaction costs and confirm that,  as $\beta \downarrow 0$,  the value function as well as $c_1^*$ and $c_2^*$  converge to the ones obtained for the no-transaction cost case in \cite{Bayraktar_2012}.

The rest of the paper is organized as follows.  Section \ref{section_model} gives a mathematical model of the problem.  In Section  \ref{section_c_1_c_2},  we compute the expected NPV of dividends under the $(c_1,c_2)$-policy via the scale function.  Section \ref{section_candidate} shows the existence of $0 \leq c_1^* < c_2^* < \infty$ that maximize the expected NPV over $c_1$ and $c_2$.  Section \ref{section_verification} verifies the optimality of the $(c_1^*,c_2^*)$-policy.  We conclude the paper with numerical results in Section \ref{section_numerics}.

\section{Mathematical Formulation} \label{section_model}
We will denote the surplus of a company by a spectrally positive \lev process $X = \left\{X_t; t \geq 0 \right\}$ whose \emph{Laplace exponent} is given by
\begin{align}
\psi(s)  := \log \E \left[ e^{-s X_1} \right] =  c s +\frac{1}{2}\sigma^2 s^2 + \int_{(0,\infty)} (e^{-s z}-1 + s z 1_{\{0 < z < 1\}}) \nu (\diff z), \quad s \in \mathbb{R} \label{laplace_spectrally_positive}
\end{align}
where $\nu$ is a \lev measure with the support $(0,\infty)$ that satisfies the integrability condition $\int_{(0,\infty)} (1 \wedge z^2) \nu(\diff z) < \infty$.  It has paths of bounded variation if and only if $\sigma = 0$ and $\int_{( 0,1)}z\, \nu(\diff z) < \infty$. In this case, we write \eqref{laplace_spectrally_positive} as
\begin{align*}
\psi(s)   =  \mathfrak{d} s +\int_{(0,\infty)} (e^{-s z}-1 ) \nu (\diff z), \quad s \in \mathbb{R} 
\end{align*}
with $\mathfrak{d} := c + \int_{( 0,1)}z\, \nu(\diff z)$; the resulting drift of the process is $-\mathfrak{d}$. We exclude the trivial case in which $X$ is a subordinator (i.e., $X$ has monotone paths a.s.). This assumption implies that $\mathfrak{d} > 0$ when $X$ is of bounded variation. 

Let $\mathbb{P}_x$ be the conditional probability under which $X_0 = x$ (also let $\mathbb{P} \equiv \mathbb{P}_0$), and let $\mathbb{F} := \left\{ \mathcal{F}_t: t \geq 0 \right\}$ be the filtration generated by $X$.  Using this, the drift of $X$ is given by
\begin{align}
\mu := \E [X_1]  = - \psi'(0+). \label{drift}
\end{align}
In order to make sure the problem is non-trivial and well-defined, we assume throughout the paper that this is finite.
\begin{assump}  \label{assump_finiteness_mu}We assume that $\mu \in (-\infty, \infty)$.
\end{assump}

A (dividend) \emph{strategy} $\pi := \left\{ L_t^{\pi}; t \geq 0 \right\}$ is given by a nondecreasing, right-continuous and $\mathbb{F}$-adapted \emph{pure jump} process starting at zero in the form $L_t^\pi = \sum_{0 \leq s \leq t} \Delta L_s^\pi$ with $\Delta L_t = L_t - L_{t-}$, $t \geq 0$. Corresponding to every strategy $\pi$, we associate a \emph{controlled surplus} process $U^\pi = \{U_t^\pi: t \geq 0 \}$, which is defined by
\begin{align*}
U_t^\pi := X_t - L_t^\pi, \quad t \geq 0,
\end{align*}
where $U^\pi_{0-} = x$ is the initial surplus and $L^\pi_{0-} = 0$.
The time of  ruin is defined to be
\begin{align*}
\sigma^\pi := \inf \left\{ t > 0: U_t^\pi < 0 \right\}.
\end{align*}
A lump-sum payment cannot be more than the available funds and hence it is required that
\begin{align}
\Delta L_t^\pi \leq U_{t-}^\pi + \Delta X_t,  \quad t \leq \sigma^\pi \; \; a.s.  \label{admissibility}
\end{align}
Let $\Pi$ be the set of all admissible strategies satisfying (\ref{admissibility}).
The problem is to compute, for $q > 0$,   the expected NPV of dividends until ruin
\begin{align*}
v_\pi (x) := \E_x \Big[ \int_0^{\sigma^\pi} e^{-qt} \diff \Big( L_t^\pi - \sum_{0 \leq s \leq t} \beta 1_{\{ \Delta L_s^\pi > 0 \}}\Big)\Big], \quad x \geq 0,
\end{align*}
where $\beta > 0$ is the unit transaction cost, and to obtain an admissible strategy that maximizes it, if such a strategy exists.  Hence the (optimal) value function is written as
\begin{equation}\label{eq:classical-p}
  v(x):=\sup_{\pi\in \Pi}v_\pi(x), \quad x \geq 0.
\end{equation}

\section{The $(c_1,c_2)$-policy} \label{section_c_1_c_2}

We aim to prove that a $(c_1^*,c_2^*)$-policy is optimal for some $c^*_2 > c_1^* \geq 0$.  For  $c_2 > c_1 \geq 0$, a $(c_1,c_2)$-policy, $\pi_{c_1,c_2} := \left\{ L_t^{c_1,c_2}; t \geq 0 \right\}$, brings the level of the controlled surplus process $U^{c_1,c_2} := X - L^{c_1,c_2}$ down to $c_1$ whenever it reaches or exceeds $c_2$.  Let us define the corresponding expected NPV of dividends as
\begin{align}
v_{c_1,c_2} (x) := \E_x \left[ \int_0^{\sigma_{c_1,c_2}} e^{-qt} \diff \Big( L_t^{c_1,c_2} - \sum_{0 \leq s \leq t} \beta 1_{\{ \Delta L_s^{c_1,c_2} > 0 \}}\Big)\right], \quad x \geq 0, \label{def_v_c}
\end{align}
where $\sigma_{c_1,c_2} := \inf \left\{ t > 0: U_t^{c_1,c_2} < 0 \right\}$ is the corresponding ruin time.  In this section, we shall express these in terms of the scale function.

\subsection{Scale functions}
Fix $q > 0$. For any spectrally positive \lev process, there exists a function called  the  \emph{q-scale function} 
\begin{align*}
W^{(q)}: \R \rightarrow [0,\infty), 
\end{align*}
which is zero on $(-\infty,0)$, continuous and strictly increasing on $[0,\infty)$, and is characterized by the Laplace transform:
\begin{align*}
\int_0^\infty e^{-s x} W^{(q)}(x) \diff x = \frac 1
{\psi(s)-q}, \qquad s > \lapinv,
\end{align*}
where
\begin{equation}
\lapinv :=\sup\{\lambda \geq 0: \psi(\lambda)=q\}. \notag
\end{equation}
Here, the Laplace exponent $\psi$ in \eqref{laplace_spectrally_positive} is known to be zero at the origin and convex on $[0,\infty)$; therefore $\lapinv$ is well-defined and is strictly positive as $q > 0$.   We also define, for $x \in \R$,
\begin{align*}
\overline{W}^{(q)}(x) &:=  \int_0^x W^{(q)}(y) \diff y, \\
Z^{(q)}(x) &:= 1 + q \overline{W}^{(q)}(x),  \\
\overline{Z}^{(q)}(x) &:= \int_0^x Z^{(q)} (z) \diff z = x + q \int_0^x \int_0^z W^{(q)} (w) \diff w \diff z.
\end{align*}
Notice that because $W^{(q)}$ is uniformly zero on the negative half line, we have
\begin{align}
Z^{(q)}(x) = 1  \quad \textrm{and} \quad \overline{Z}^{(q)}(x) = x, \quad x \leq 0.  \label{z_below_zero}
\end{align}

Let us define the \emph{first down-} and \emph{up-crossing times}, respectively, by
\begin{align}
\label{first_passage_time}
\tau_a^- := \inf \left\{ t \geq 0: X_t < a \right\} \quad \textrm{and} \quad \tau_b^+ := \inf \left\{ t \geq 0: X_t >  b \right\}, \quad a,b \in \R.
\end{align}
Then we have for any $b > 0$
\begin{align}
\begin{split}
\E_x \left[ e^{-q \tau_0^-} 1_{\left\{ \tau_b^+ > \tau_0^- \right\}}\right] &= \frac {W^{(q)}(b-x)}  {W^{(q)}(b)}, \\
 \E_x \left[ e^{-q \tau_b^+} 1_{\left\{ \tau_b^+ < \tau_0^- \right\}}\right] &= Z^{(q)}(b-x) -  Z^{(q)}(b) \frac {W^{(q)}(b-x)}  {W^{(q)}(b)}.
\end{split}
 \label{laplace_in_terms_of_z}
\end{align}
Notice for the case of spectrally negative \lev process starting at $x$, analogous results hold by replacing $b-x$ with $x$.

Fix $a \geq 0$ and define $\psi_a(\cdot)$ as the Laplace exponent of $X$ under $\p^a$ with the change of measure 
\begin{align}
\left. \frac {\diff \p^a} {\diff \p}\right|_{\mathcal{F}_t} = \exp(a X_t - \psi(a) t), \quad t \geq 0; \label{change_of_measure}
\end{align}
see page 213 of \cite{Kyprianou_2006}. It is given for all $s > -a$ by
\begin{align*}
\psi_a(s) :=\Big(  a \sigma^2  + c - \int_0^1 u (e^{-au}-1) \nu(\diff u)\Big) s
+\frac{1}{2}\sigma^2 s^2 +\int_0^\infty (e^{- s u}-1 + s u 1_{\{ 0 < u < 1 \}}) e^{-a u}\,\nu(\diff u).
\end{align*}
 If $W_a^{(q)}$ and $Z_a^{(q)}$ are the scale functions associated with $X$ under $\p^a$ (or equivalently with $\psi_a(\cdot)$).  
Then, by Lemma 8.4 of \cite{Kyprianou_2006},
\begin{align}
W_a^{(q-\psi(a))}(x) = e^{-a x} W^{(q)}(x), \quad x \in \R, \label{scale_measure_change}
\end{align}
which is well-defined even for $q \leq \psi(a)$ by Lemmas 8.3 and 8.5 of \cite{Kyprianou_2006}.

\begin{remark} \label{remark_smoothness_zero}
\begin{enumerate}
\item If $X$ is of unbounded variation, it is known that $W^{(q)}$ is $C^1(0,\infty)$; see, e.g.,\ Chan et al.\ \cite{Chan_2009}.  Hence, 
\begin{enumerate}
\item $Z^{(q)}$ is $C^1(0,\infty)$ and $C^0 (\R)$ for the bounded variation case, while it is $C^2(0,\infty)$ and $C^1 (\R)$ for the unbounded variation case, and
\item $\overline{Z}^{(q)}$ is $C^2(0,\infty)$ and $C^1 (\R)$ for the bounded variation case, while it is $C^3(0,\infty)$ and $C^2 (\R)$ for the unbounded variation case.
\end{enumerate}
\item Regarding the asymptotic behavior near zero, we have that
\begin{equation}\label{eq:Wq0}
W^{(q)} (0) = \left\{ \begin{array}{ll} 0, & \textrm{if $X$ is of unbounded
variation,} \\ \frac 1 {\mathfrak{d}}, & \textrm{if $X$ is of bounded variation,}
\end{array} \right.
\end{equation}
and 
\begin{equation}\label{eq:Wqp0}
W^{(q)'} (0+) := \lim_{x \downarrow 0}W^{(q)'} (x) =
\left\{ \begin{array}{ll}  \frac 2 {\sigma^2}, & \textrm{if }\sigma > 0, \\
\infty, & \textrm{if }\sigma = 0 \; \textrm{and} \; \nu(0,\infty) = \infty, \\
\frac {q + \nu(0,\infty)} {\mathfrak{d}^2}, & \textrm{if $X$ is compound Poisson.}
\end{array} \right.
\end{equation}
\item As in (8.18) and Lemma 8.2 of \cite{Kyprianou_2006},
\begin{align*}
 \frac {W^{(q)'}(y)} {W^{(q)}(y)} \leq \frac {W^{(q)'}(x)} {W^{(q)}(x)},  \quad  y > x > 0,
\end{align*}
where $W^{(q)'}$ is understood as the right-derivative if it is not differentiable.  In all cases, $W^{(q)'}(x-) \geq W^{(q)'}(x+)$ for all $x > 0$.
\end{enumerate}
\end{remark}

\subsection{The expected NPV of dividends for the $(c_1,c_2)$-policy}
Now we obtain \eqref{def_v_c} using the scale function.
By the strong Markov property,
it must satisfy, for every $ 0 \leq x < c_2$ and $0 \leq c_1 < c_2$,
\begin{align}
v_{c_1,c_2} (x) = \E_x \left[ e^{-q \tau_{c_2}^+} 1_{\{ \tau_{c_2}^+ < \tau_0^-\}} (X_{\tau_{c_2}^+} - c_1 - \beta)\right] + \E_x \left[ e^{-q \tau_{c_2}^+} 1_{\{ \tau_{c_2}^+ < \tau_0^-\}}\right] \bar{v}_{c_1,c_2},   \label{def_v_c_exp}
\end{align}
where $\bar{v}_{c_1,c_2} := v_{c_1,c_2} (c_1)$.
Solving for $x = c_1$, we have
\begin{align}
\bar{v}_{c_1,c_2}= \frac {\E_{c_1} \left[ e^{-q \tau_{c_2}^+} 1_{\{ \tau_{c_2}^+ < \tau_0^-\}} (X_{\tau_{c_2}^+} - c_1 - \beta)\right]} {1- \E_{c_1} \left[ e^{-q \tau_{c_2}^+} 1_{\{ \tau_{c_2}^+ < \tau_0^-\}} \right]}, \quad 0 \leq c_1 < c_2. \label{v_bar_def}
\end{align}

 The Laplace transform $\E_x \big[ e^{-q \tau_{c}^+ - v X_{\tau_{c}^+}} 1_{\{ \tau_{c}^+ < \tau_0^-\}} \big]$, $q,v > 0$, was computed in Corollary 3 of \cite{Ivanovs_Palmowski}.  The following result is the derivative of this transform at $v = 0$.
\begin{lemma} \label{lemma_h}
For $0 \leq x < c$,
\begin{align*}
\E_x \left[ e^{-q \tau_{c}^+} 1_{\{ \tau_{c}^+ < \tau_0^-\}} X_{\tau_{c}^+}\right]
&= -R^{(q)}(c-x)  + \left( c - \frac {\mu} q \right) Z^{(q)}(c-x) \\ &-  \left[ \left(c -\frac {\mu} q  \right) Z^{(q)}(c) - R^{(q)}(c) \right] \frac {W^{(q)}(c-x)} {W^{(q)}(c)},
\end{align*}
where
\begin{align*}
R^{(q)}(y) := \overline{Z}^{(q)}(y) - \frac {\mu} q, \quad y \in \R.
\end{align*}
\end{lemma}

By this lemma, \eqref{laplace_in_terms_of_z} and \eqref{v_bar_def}, 
 we can write
\begin{align}
\bar{v}_{c_1,c_2} = \frac {f(c_1, c_2) } {g(c_1, c_2)}, \quad 0 \leq c_1 < c_2, \label{v_as_f_g}
\end{align}
where
\begin{align} \label{f_c}
\begin{split}
f(c_1, c_2)&:=   -R^{(q)}(c_2-c_1)  + \left( c_2 -\frac {\mu} q \right) Z^{(q)}(c_2-c_1)   \\ &-  \left[ \left( c_2 -  \frac {\mu} q \right) Z^{(q)}(c_2) - R^{(q)}(c_2) \right] \frac {W^{(q)}(c_2-c_1)} {W^{(q)}(c_2)} \\
&- (c_1 + \beta) \left[ Z^{(q)}(c_2-c_1) - Z^{(q)}(c_2) \frac {W^{(q)}(c_2-c_1)}  {W^{(q)}(c_2)} \right] \\
&=   -R^{(q)}(c_2-c_1)  + \left( c_2 - c_1-  \beta - \frac {\mu} q \right) Z^{(q)}(c_2-c_1)   \\ &-  \left[ \left(  c_2 - c_1 - \beta -\frac {\mu} q \right) Z^{(q)}(c_2) - R^{(q)}(c_2) \right] \frac {W^{(q)}(c_2-c_1)} {W^{(q)}(c_2)}
\end{split}
\end{align}
and
\begin{align} \label{g_c}
g(c_1, c_2) &:= 1- Z^{(q)}(c_2-c_1) + Z^{(q)}(c_2) \frac {W^{(q)}(c_2-c_1)}  {W^{(q)}(c_2)}.
\end{align}

\section{Candidate strategies} \label{section_candidate}
Using the results in the previous section, we now have an analytical expression for \eqref{def_v_c} or equivalently \eqref{def_v_c_exp}.  For $0 \leq x < c_2$ and $0 \leq c_1 < c_2$, this expression reduces to
\begin{align} \label{v_less_than_c_2}
\begin{split}
v_{c_1,c_2} (x) 
 &= -R^{(q)}(c_2-x)  + \left( c_2 - \frac {\mu} q \right) Z^{(q)}(c_2-x) \\ &-  \left[ \left(  c_2 -\frac {\mu} q \right) Z^{(q)}(c_2) - R^{(q)}(c_2) \right] \frac {W^{(q)}(c_2-x)} {W^{(q)}(c_2)} \\ &+ (\bar{v}_{c_1,c_2} - c_1 -\beta) \left[Z^{(q)}(c_2-x) - Z^{(q)}(c_2) \frac {W^{(q)}(c_2-x)}  {W^{(q)}(c_2)} \right] \\
&= -R^{(q)}(c_2-x)  + \gamma(c_1,c_2)  Z^{(q)}(c_2-x)  - G(c_1,c_2) \frac {W^{(q)}(c_2-x)} {W^{(q)}(c_2)},
\end{split}
\end{align}
where
\begin{align} \label{definition_gamma_G}
\begin{split}
\gamma(c_1,c_2) &:= \bar{v}_{c_1,c_2}  + c_2 - c_1 -\beta  -\frac {\mu} q, \\
G(c_1,c_2) &:= \gamma(c_1,c_2) Z^{(q)}(c_2) - R^{(q)}(c_2).
\end{split}
\end{align}
For $x \geq c_2$, we have
\begin{align}
v_{c_1,c_2} (x) = x - c_1 - \beta + \bar{v}_{c_1,c_2}. \label{v_above_c_2}
\end{align}

In view of \eqref{v_above_c_2}, a necessary condition for a $(c_1,c_2)$-policy to be optimal is that $c_1$ and $c_2$ maximize $\bar{v}_{c_1,c_2} - c_1$.  In this section, we first obtain the first-order conditions by computing its partial derivatives with respect to $c_1$ and $c_2$ and then show the existence of finite-valued maximizers. 
%
In the rest of the paper, the derivative is understood as the right-derivative when the scale function $W^{(q)}$ fails to be differentiable on $(0,\infty)$.

\subsection{First-order conditions}

\begin{lemma} \label{lemma_derivative_v_c_2}
For every $0 \leq c_1 < c_2$, 
\begin{align*}
\frac \partial {\partial c_2}(\bar{v}_{c_1,c_2} -c_1) = \frac \partial {\partial c_2}\bar{v}_{c_1,c_2}
= -\frac {G(c_1,c_2)} {g(c_1,c_2)}\frac  \partial {\partial c_2}\frac {W^{(q)}(c_2-c_1)}  {W^{(q)}(c_2)}.
\end{align*}
\end{lemma}
\begin{proof}
Differentiating  \eqref{f_c}, we obtain
\begin{align*}
\frac \partial {\partial c_2}f(c_1, c_2)
&=   -Z^{(q)}(c_2-c_1)  + Z^{(q)}(c_2-c_1) + \left( c_2 - c_1-  \beta - \frac {\mu} q \right) q W^{(q)}(c_2-c_1)   \\ &-  \left[ \left(  c_2 - c_1 - \beta - \frac {\mu} q\right) q W^{(q)}(c_2) + Z^{(q)}(c_2) -Z^{(q)}(c_2) \right] \frac {W^{(q)}(c_2-c_1)} {W^{(q)}(c_2)} \\
 &-  \left[ \left( c_2 - c_1 - \beta -\frac {\mu} q \right) Z^{(q)}(c_2) - R^{(q)}(c_2) \right] \frac \partial {\partial c_2}\frac {W^{(q)}(c_2-c_1)}  {W^{(q)}(c_2)} \\
&=   -  \left[ \left(  c_2 - c_1 - \beta -\frac {\mu} q \right) Z^{(q)}(c_2) - R^{(q)}(c_2) \right] \frac \partial {\partial c_2}\frac {W^{(q)}(c_2-c_1)}  {W^{(q)}(c_2)}.
\end{align*}
On the other hand, differentiating \eqref{g_c} yields
\begin{align*}
\frac \partial {\partial c_2}g(c_1,c_2) &=  -q W^{(q)}(c_2-c_1) + q W^{(q)}(c_2) \frac {W^{(q)}(c_2-c_1)}  {W^{(q)}(c_2)} +Z^{(q)}(c_2) \frac \partial {\partial c_2}\frac {W^{(q)}(c_2-c_1)}  {W^{(q)}(c_2)} \\
&= Z^{(q)}(c_2) \frac \partial {\partial c_2}\frac {W^{(q)}(c_2-c_1)}  {W^{(q)}(c_2)}.
\end{align*}
Using the last two equations along with \eqref{v_as_f_g}, we have
\begin{align*}
&g(c_1,c_2) \frac \partial {\partial c_2} \bar{v}_{c_1,_2}= \frac \partial {\partial c_2} f(c_1,c_2) -\bar{v}_{c_1,c_2}\frac \partial {\partial c_2} g(c_1,c_2) \\
&= -  \left[ \left(  c_2 - c_1 - \beta -\frac {\mu} q \right) Z^{(q)}(c_2) - R^{(q)}(c_2) \right] \frac \partial {\partial c_2}\frac {W^{(q)}(c_2-c_1)}  {W^{(q)}(c_2)} - \bar{v}_{c_1,c_2} Z^{(q)}(c_2) \frac \partial {\partial c_2}\frac {W^{(q)}(c_2-c_1)}  {W^{(q)}(c_2)} \\
&= -G(c_1,c_2) \frac  \partial {\partial c_2}\frac {W^{(q)}(c_2-c_1)}  {W^{(q)}(c_2)}.
\end{align*}
\end{proof}

\begin{lemma} \label{lemma_derivative_c_1}
For $0 < c_1 < c_2$,
\begin{align*}
&\frac \partial {\partial c_1} ( \bar{v}_{c_1,c_2} - c_1 ) = \frac \partial {\partial c_1} \left(\frac {f(c_1, c_2) - c_1 g(c_1,c_2)} {g(c_1,c_2)}  \right) 
= \frac 1 {g(c_1,c_2) } \Big[ -H(c_1,c_2)  + G(c_1,c_2)  \frac {W^{(q)'}(c_2-c_1)} {W^{(q)}(c_2)} \Big],
\end{align*}
where
\begin{align*}
H(c_1,c_2) :=q \left[ \gamma(c_1,c_2) W^{(q)}(c_2-c_1) - \overline{W}^{(q)}(c_2-c_1) \right]. 
\end{align*}
\end{lemma}
\begin{proof}
By \eqref{f_c} and \eqref{g_c},
\begin{align*}
 f(c_1, c_2) - c_1 g(c_1,c_2) &=   -R^{(q)}(c_2-c_1)  + \left( c_2 - c_1-  \beta - \frac {\mu} q \right) Z^{(q)}(c_2-c_1)   \\ &-  \left[ \left( c_2 - c_1 - \beta  -\frac {\mu} q  \right) Z^{(q)}(c_2) - R^{(q)}(c_2) \right] \frac {W^{(q)}(c_2-c_1)} {W^{(q)}(c_2)} \\
&- c_1 + c_1 Z^{(q)}(c_2-c_1)  - c_1 Z^{(q)}(c_2) \frac {W^{(q)}(c_2-c_1)}  {W^{(q)}(c_2)} \\
&=   -R^{(q)}(c_2-c_1)  - c_1 + \left(  c_2  - \beta  -\frac {\mu} q \right)  Z^{(q)}(c_2-c_1) \\ &-  \left[ \left(  c_2  - \beta -\frac {\mu} q \right) Z^{(q)}(c_2) - R^{(q)}(c_2)  \right] \frac {W^{(q)}(c_2-c_1)} {W^{(q)}(c_2)},
\end{align*}
and hence its derivative equals
\begin{align*}
 \frac \partial {\partial c_1}[f(c_1, c_2) - c_1 g(c_1,c_2)]  &= q \overline{W}^{(q)}(c_2-c_1) -   \left(  c_2  - \beta -\frac {\mu} q \right) q W^{(q)}(c_2-c_1)   \\ &+  \left[ \left(  c_2  - \beta -\frac {\mu} q \right) Z^{(q)}(c_2) - R^{(q)}(c_2)  \right]   \frac {W^{(q)'}(c_2-c_1)} {W^{(q)}(c_2)}.
\end{align*}
Because $\partial g(c_1, c_2)/\partial c_1 =  q W^{(q)}(c_2-c_1) - Z^{(q)}(c_2) \frac {W^{(q)'}(c_2-c_1)}  {W^{(q)}(c_2)}$ and by \eqref{v_as_f_g},
\begin{align*}
&g(c_1,c_2) \frac \partial {\partial c_1} \left(\frac {f(c_1, c_2) - c_1 g(c_1,c_2)} {g(c_1,c_2)}  \right) = \frac \partial {\partial c_1}[f(c_1, c_2) - c_1 g(c_1,c_2)] - (\bar{v}_{c_1,c_2} - c_1)   \frac \partial {\partial c_1}{g(c_1,c_2)}  \\
&= q \overline{W}^{(q)}(c_2-c_1) -   \left(  c_2  - \beta  -\frac {\mu} q \right) q W^{(q)}(c_2-c_1)   +  \left[ \left(  c_2  - \beta -\frac {\mu} q \right) Z^{(q)}(c_2) - R^{(q)}(c_2)  \right]  \frac {W^{(q)'}(c_2-c_1)} {W^{(q)}(c_2)} \\
&- (\bar{v}_{c_1,c_2} - c_1)  \left[ q W^{(q)}(c_2-c_1) - Z^{(q)}(c_2) \frac {W^{(q)'}(c_2-c_1)}  {W^{(q)}(c_2)} \right] \\
&= -H(c_1,c_2)  + G(c_1,c_2) \frac {W^{(q)'}(c_2-c_1)} {W^{(q)}(c_2)}.
\end{align*}
\end{proof}

\begin{remark}
The first-order conditions obtained above are for \eqref{v_above_c_2}.  However, these are in fact the same for \eqref{v_less_than_c_2} for any $0 \leq x < c_2$.  Differentiating the first equality of  \eqref{v_less_than_c_2}, 
\begin{align}
\frac \partial {\partial c_1} v_{c_1,c_2} (x) &=  \frac \partial {\partial c_1} (\bar{v}_{c_1,c_2} - c_1 -\beta) \left[Z^{(q)}(c_2-x) - Z^{(q)}(c_2) \frac {W^{(q)}(c_2-x)}  {W^{(q)}(c_2)} \right], \quad 0 < c_1 < c_2,
\end{align}
whose sign is the same as  that of $ \partial  (\bar{v}_{c_1,c_2} - c_1) / {\partial c_1}$ thanks to \eqref{laplace_in_terms_of_z} which guarantees that the expression inside the square brackets is positive.  Moreover, by differentiating  \eqref{v_less_than_c_2} and by Lemma \ref{lemma_derivative_v_c_2}, for $0 \leq c_1 < c_2$,
\begin{align*}
\frac \partial {\partial c_2}  v_{c_1,c_2} (x) 
&= - G(c_1,c_2) \frac \partial {\partial c_2} \frac {W^{(q)}(c_2-x)} {W^{(q)}(c_2)} + \left[Z^{(q)}(c_2-x) - Z^{(q)}(c_2) \frac {W^{(q)}(c_2-x)}  {W^{(q)}(c_2)} \right] \frac \partial {\partial c_2}\bar{v}_{c_1,c_2} \\
&= - G(c_1,c_2)  \left( 1 + \frac  {Z^{(q)}(c_2-x) - Z^{(q)}(c_2) \frac {W^{(q)}(c_2-x)}  {W^{(q)}(c_2)}} {g(c_1,c_2)}\right) \frac \partial {\partial c_2} \frac {W^{(q)}(c_2-x)} {W^{(q)}(c_2)},
\end{align*}
whose sign is the same as that of $\partial \bar{v}_{c_1,c_2} / {\partial c_2}$ due to item (3) of Remark~\ref{remark_smoothness_zero}.
\end{remark}

\subsection{Existence and some properties of maximizers}
Now we are ready to show that the maximizers of $\bar{v}_{c_1,c_2} - c_1$ exist. We will also describe equations that can be used to identify these points.

\begin{lemma} \label{lemma_sup_bounded_set}
We have $\sup_{0 \leq c_1 < c_2}(\bar{v}_{c_1,c_2} - c_1) = \sup_{0 \leq c_1 < c_2 \leq C}(\bar{v}_{c_1,c_2} - c_1)$ for sufficiently large $C < \infty$.
\end{lemma}
\begin{proof}
By Lemma \ref{lemma_h} and  \eqref{laplace_in_terms_of_z}, for any $c_2 > c_1 \geq 0$,
\begin{align*} 
\E_{c_1} \left[ e^{-q \tau_{c_2}^+} 1_{\{ \tau_{c_2}^+ < \tau_0^-\}} (X_{\tau_{c_2}^+} - \beta )\right] &= - R^{(q)}(c_2-c_1) + \left(c_2 - \beta - \frac \mu q \right) Z^{(q)}(c_2-c_1) \\ &- \left[ \left(c_2 - \beta - \frac \mu q \right) Z^{(q)}(c_2) - R^{(q)}(c_2)  \right] \frac {W^{(q)}(c_2-c_1)} {W^{(q)}(c_2)},
\end{align*}
and hence
\begin{align} \label{expectation_derivative_c_2}
\begin{split}
\frac \partial {\partial c_2}\E_{c_1} \left[ e^{-q \tau_{c_2}^+} 1_{\{ \tau_{c_2}^+ < \tau_0^-\}} (X_{\tau_{c_2}^+} - \beta )\right] &=- \left[ \left( c_2 - \beta - \frac \mu q \right) Z^{(q)}(c_2) - R^{(q)}(c_2) \right] \frac \partial {\partial c_2} \frac {W^{(q)}(c_2-c_1)} {W^{(q)}(c_2)} \\
&=- A(c_2) W^{(q)}(c_2-c_1) \left( \frac {W^{(q)'}(c_2-c_1)}  {W^{(q)}(c_2-c_1)} - \frac {W^{(q)'}(c_2)}  {W^{(q)}(c_2)} \right),
\end{split}
\end{align}
where $A(c) := \left( c - \beta - \frac \mu q \right) \frac {Z^{(q)}(c)} {W^{(q)}(c)} - \frac {R^{(q)}(c)} {W^{(q)}(c)}$, $c > 0$.
 It follows from Exercise 8.5 of \cite{Kyprianou_2006} and Proposition 2 of \cite{Avram_et_al_2007} that $Z^{(q)}(c)/W^{(q)}(c) \rightarrow q/\Phi(q) \in (0,\infty)$ and $R^{(q)}(c)/W^{(q)}(c) \rightarrow q/\Phi(q)^{2} \in (0,\infty)$ as $c \uparrow \infty$, respectively. As a result,  $A(c) \uparrow \infty$ and hence there exists $B < \infty$ such that  
\begin{align}
A(c)> 0, \quad c \geq B. \label{bound_B}
\end{align}
Now because $\frac {W^{(q)'}(c_2-c_1)}  {W^{(q)}(c_2-c_1)} - \frac {W^{(q)'}(c_2)}  {W^{(q)}(c_2)}  > 0$ by Remark  \ref{remark_smoothness_zero}(3), we have $ \partial \E_{c_1} \left[ e^{-q \tau_{c_2}^+} 1_{\{ \tau_{c_2}^+ < \tau_0^-\}} (X_{\tau_{c_2}^+} - \beta )\right] / {\partial c_2} < 0$ for any $c_2 > c_1 \geq B$.  Hence for any fixed $c_1 \geq B$,
\begin{align*}
\sup_{c_2: c_2 > c_1}\E_{c_1} \left[ e^{-q \tau_{c_2}^+} 1_{\{ \tau_{c_2}^+ < \tau_0^-\}} (X_{\tau_{c_2}^+} - \beta )\right] = \lim_{c_2 \downarrow c_1}\E_{c_1} \left[ e^{-q \tau_{c_2}^+} 1_{\{ \tau_{c_2}^+ < \tau_0^-\}} (X_{\tau_{c_2}^+} - \beta )\right] = (c_1 - \beta) - A(c_1) W^{(q)}(0).
\end{align*}
Now by the definition of $\bar{v}_{c_1,c_2}$ as in \eqref{v_bar_def}, for any fixed $c_1 \geq B$,
\begin{align*}
\sup_{c_2: c_2 > c_1}(\bar{v}_{c_1,c_2}- c_1) &= \sup_{c_2:c_2 > c_1} \frac {- c_1 + \E_{c_1} \left[ e^{-q \tau_{c_2}^+} 1_{\{ \tau_{c_2}^+ < \tau_0^-\}} (X_{\tau_{c_2}^+}  - \beta)\right]} {1- \E_{c_1} \left[ e^{-q \tau_{c_2}^+} 1_{\{ \tau_{c_2}^+ < \tau_0^-\}} \right]} \\ &\leq \sup_{c_2:c_2 > c_1} \frac {- c_1 +\sup_{c_2: c_2 > c_1}\E_{c_1} \left[ e^{-q \tau_{c_2}^+} 1_{\{ \tau_{c_2}^+ < \tau_0^-\}} (X_{\tau_{c_2}^+} - \beta)\right]} {1- \E_{c_1} \left[ e^{-q \tau_{c_2}^+} 1_{\{ \tau_{c_2}^+ < \tau_0^-\}} \right]} \\ &\leq \sup_{c_2: c_2 > c_1} \frac { - \beta - A(c_1) W^{(q)}(0) } {1- \E_{c_1} \left[ e^{-q \tau_{c_2}^+} 1_{\{ \tau_{c_2}^+ < \tau_0^-\}} \right]},
\end{align*}
which is negative by \eqref{bound_B}.  On the other hand, because $c_1=0$ and $c_2 > 0$ attain $\bar{v}_{c_1,c_2} - c_1 = 0$,
we have
\begin{align}
\sup_{(c_1,c_2): c_2 > c_1\geq 0}(\bar{v}_{c_1,c_2}- c_1) = \sup_{(c_1,c_2):  c_2 > c_1 \geq 0, c_1 \leq  B}(\bar{v}_{c_1,c_2}- c_1). \label{sup_reduction}
\end{align}

Now fix $c_1 \leq B$ and $c_2 \geq B + \delta$ for any $\delta > 0$.  Then 
\begin{align*}
\bar{v}_{c_1,c_2}- c_1 = \frac {-c_1 + \E_{c_1} \left[ e^{-q \tau_{c_2}^+} 1_{\{ \tau_{c_2}^+ < \tau_0^-\}} (X_{\tau_{c_2}^+}  - \beta)\right]} {1- \E_{c_1} \left[ e^{-q \tau_{c_2}^+} 1_{\{ \tau_{c_2}^+ < \tau_0^-\}} \right]} \geq \frac {- B - \beta} {1- \E_{B} \left[ e^{-q \tau_{B+\delta}^+} 1_{\{ \tau_{B+\delta}^+ < \tau_0^-\}} \right]} =: M > -\infty.
\end{align*}
By Lemma \ref{lemma_derivative_v_c_2} and Remark  \ref{remark_smoothness_zero}(3), 
\begin{align*}
&\frac \partial {\partial c_2}\bar{v}_{c_1,c_2}
= -\frac {G(c_1,c_2)} {g(c_1,c_2)}\frac  \partial {\partial c_2}\frac {W^{(q)}(c_2-c_1)}  {W^{(q)}(c_2)} \\
&= - \left( \left[\bar{v}_{c_1,c_2}  + c_2 - c_1 -\beta  -\frac {\mu} q \right]  \frac {Z^{(q)}(c_2)}{W^{(q)}(c_2)}  - \frac{R^{(q)}(c_2)}{W^{(q)}(c_2)}   \right) \frac {W^{(q)}(c_2-c_1)} {g(c_1,c_2)} \left( \frac {W^{(q)'}(c_2-c_1)}  {W^{(q)}(c_2-c_1)} - \frac {W^{(q)'}(c_2)}  {W^{(q)}(c_2)} \right) \\
&\leq -\left( \left[M + c_2  -\beta  -\frac {\mu} q \right]  \frac {Z^{(q)}(c_2)}{W^{(q)}(c_2)}  - \frac{R^{(q)}(c_2)}{W^{(q)}(c_2)}   \right)  \frac {W^{(q)}(c_2-c_1)} {g(c_1,c_2)}  \left( \frac {W^{(q)'}(c_2-c_1)}  {W^{(q)}(c_2-c_1)} - \frac {W^{(q)'}(c_2)}  {W^{(q)}(c_2)} \right).
\end{align*}
Using Remark \ref{remark_smoothness_zero}(3) and the fact that $(M + c_2 -\beta  -\frac {\mu} q ) \frac {Z^{(q)}(c_2)}{W^{(q)}(c_2)}  - \frac{R^{(q)}(c_2)}{W^{(q)}(c_2)}  \xrightarrow{c_2 \uparrow \infty} \infty$,
 it follows that there exists a sufficiently large constant $C$ such that
\begin{align*}
\sup_{c_1 \leq B, c_2 \geq C}\frac \partial {\partial c_2}\bar{v}_{c_1,c_2} \leq 0.
\end{align*}
Combining the last  inequality with \eqref{sup_reduction} completes the proof.
\end{proof}

\begin{lemma} \label{lemma_G_limit}
Fix any $c_1 \geq 0$, $\lim_{c_2 \downarrow c_1}G(c_1,c_2) < 0$.
\end{lemma}
\begin{proof}
We have
\begin{align*}
\gamma (c_1,c_2) \xrightarrow{c_2 \downarrow c_1} \left\{ \begin{array}{ll} -\infty, & \textrm{if $X$ is of unbounded variation}, \\  Z^{(q)}(c_1)^{-1} \left[ -\frac {W^{(q)}(c_1)} {W^{(q)}(0)} \beta + R^{(q)}(c_1)\right], & \textrm{if $X$ is of bounded variation}. \end{array} \right.
\end{align*}
When $X$ is of unbounded variation $\lim_{c_2 \downarrow c_1}G(c_1,c_2) = -\infty$ while when $X$ is of bounded variation, by \eqref{definition_gamma_G},
 $\lim_{c_2 \downarrow c_1}G(c_1,c_2) = - \frac {W^{(q)}(c_1)} {W^{(q)}(0)} \beta < 0$.

\end{proof}

This lemma, together with Lemma \ref{lemma_derivative_v_c_2} and Remark \ref{remark_smoothness_zero}(3), implies that, for any fixed $c_1 \geq 0$, $\partial \bar{v}_{c_1,c_2}/  {\partial c_2}$ is negative near $c_1$; consequently there exist $\bar{v}_{c_1,c_1} := \lim_{c_2 \downarrow c_1} \bar{v}_{c_1,c_2}$ (which can be shown to be $-\infty$ when $X$ is of unbounded variation).   Because $\bar{v}_{c_1,c_2} - c_1$ is continuous and we have a compact domain $\{ (c_1,c_2): 0 \leq c_1 \leq c_2, 0 \leq c_2 \leq C \}$ for large $C$ by Lemma \ref{lemma_sup_bounded_set}, we have a maximum.   Furthermore, Lemmas \ref{lemma_derivative_v_c_2} and \ref{lemma_G_limit} show that if $c_1$ and $c_2$ maximize $\bar{v}_{c_1,c_2} - c_1$,  it must hold that $c_2$ is away from $c_1$.

\begin{lemma} \label{lemma_h_g}
Suppose $c_1$ and $c_2$ maximize $\bar{v}_{c_1,c_2} - c_1$.  Then $G(c_1,c_2) = 0$ and  $H(c_1,c_2) \geq 0$.   In particular, if $c_1 > 0$, we must have $H(c_1,c_2) = 0$.
\end{lemma}
\begin{proof}
By Lemmas \ref{lemma_sup_bounded_set} and  \ref{lemma_G_limit}, $c_2 \in (c_1,\infty)$.  Hence, by Lemma \ref{lemma_derivative_v_c_2}, we must have $G(c_1,c_2) = 0$.  On the other hand, by Lemma \ref{lemma_derivative_c_1},
\begin{align*}
&\frac \partial {\partial c_1} ( \bar{v}_{c_1,c_2} - c_1 ) 
= -\frac { H(c_1,c_2)} {g(c_1,c_2) }.
\end{align*}
If $H(c_1,c_2) < 0$, the derivative is positive and it violates the assumed optimality.  In particular, if $c_1 \in (0,c_2)$, then the derivative must vanish and hence $H(c_1,c_2) = 0$.
\end{proof}



Combining the above arguments, we arrive at the following proposition.

\begin{proposition} \label{proposition_existence}There exist  $(c_1,c_2)$ that maximize $\bar{v}_{c_1,c_2} - c_1$ and satisfy the following two properties.
\begin{enumerate}
\item $0 < c_2 < \infty$ and  $G(c_1,c_2) = 0$;
\item either $0 < c_1 < c_2$ with  $H(c_1,c_2) = 0$,
or $c_1 = 0$ with $H(0,c_2) \geq 0$.
\end{enumerate}
\end{proposition}

%
%
%

\begin{remark} \label{remark_gamma_positive} Suppose $c_1$ and $c_2$ are such that $H(c_1,c_2) \geq 0$ and $G(c_1,c_2) = 0$.  Then, $\gamma(c_1,c_2) > 0$.  To see why this is so, by Lemma \ref{lemma_G_limit}, $G(c_1,c_2) = 0$ implies that $c_1 < c_2$ and, together with $H(c_1,c_2) \geq 0$, we have $\gamma(c_1,c_2) \geq \overline{W}^{(q)}(c_2-c_1)/W^{(q)}(c_2-c_1)> 0$.
\end{remark}

\section{Verification of optimality} \label{section_verification}
By Proposition \ref{proposition_existence}, there exist $0 \leq c_1^* < c_2^* < \infty$ such that $G(c_1^*,c_2^*) = 0$ and either 
\begin{description}
\item[Case 1] $c_1^* > 0$ with $H(c_1^*,c_2^*) = 0$,
or 
\item[Case 2] $c_1^* = 0$ with $H(0,c_2^*) \geq 0$.
\end{description}
We will show that such a $(c_1^*,c_2^*)$-policy describes an optimal policy (and as a result the conditions written in terms of $H$ and $G$ are both necessary and sufficient for $(c_1^*,c_2^*)$ to be optimal.)
Propositions \ref{verification_1}  and \ref{verification_2}  will play a key role.

By substituting $G(c_1^*,c_2^*) = 0$ in \eqref{v_less_than_c_2}, 
\begin{align*} 
v_{c_1^*,c_2^*} (x) 
= \left\{ \begin{array}{ll} -R^{(q)}(c_2^*-x)  + \gamma(c_1^*,c_2^*)  Z^{(q)}(c_2^*-x), & 0 \leq x < c_2^*, \\x - c_1^* - \beta + \bar{v}_{c_1^*,c_2^*}, & x \geq c_2^*. \end{array} \right.
\end{align*}
In fact, by \eqref{z_below_zero} and by the definition of $\gamma(c_1^*,c_2^*)$ as in \eqref{definition_gamma_G},  we can write for any $x \geq 0$,
\begin{align} \label{value_function}
v_{c_1^*,c_2^*} (x) 
=  -R^{(q)}(c_2^*-x)  + \gamma(c_1^*,c_2^*)  Z^{(q)}(c_2^*-x).
\end{align}
It is clear that it is continuous at $c_2^*$.  Regarding its differentiability, we have
\begin{align} \label{v_prime}
\begin{split}
v_{c_1^*,c_2^*}' (x) 
&=  Z^{(q)}(c_2^*-x)  -  \gamma (c_1^*,c_2^*) q W^{(q)}(c_2^*-x) ,
\end{split}
\end{align}
whose limit equals
\begin{align} \label{v_prime_limit}
v_{c_1^*,c_2^*}' (c_2^*-)
&= 1  - \gamma (c_1^*,c_2^*) q W^{(q)}(0).
\end{align}
Because $v_{c_1^*,c_2^*}' (c_2^*+) = 1$, the differentiability at $c_2^*$ is satisfied if and only if $X$ is of unbounded variation by \eqref{eq:Wq0} and Remark  \ref{remark_gamma_positive}.  We summarize these observations in the lemma below.
\begin{lemma}[smoothness at $c_2^*$] \label{lemma_smoothness_c_2} The function $v_{c_1^*,c_2^*} (\cdot)$ is continuous (resp.\ differentiable) at $c_2^*$ when $X$ is of bounded (resp.\ unbounded) variation.
\end{lemma}

\begin{remark}
Differentiating \eqref{v_prime} further,
\begin{align}
v_{c_1^*,c_2^*}'' (x) 
&= - q W^{(q)}(c_2^*-x)  + \gamma (c_1^*,c_2^*) q W^{(q)'}(c_2^*-x),  \label{v_double_prime}
\end{align}
for a.e.\ $x \in (0,c_2^*)$ and its limit as $x \uparrow c_2^*$ equals
\begin{align} \label{v_double_prime_limit}
v_{c_1^*,c_2^*}'' (c_2^*-) 
&=  - q W^{(q)}(0)  + \gamma (c_1^*,c_2^*) q W^{(q)'}(0+).
\end{align}
These results on the second derivative are used in deriving Propositions \ref{verification_1}  and  \ref{verification_2}  below.
\end{remark}

By Remark \ref{remark_smoothness_zero}(1) and Lemma \ref{lemma_smoothness_c_2}, the function $v_{c_1^*,c_2^*}$ is $C^0(0,\infty)$ and $C^1((0,\infty) \backslash \{ c_2^*\}) $ (resp.\ $C^1(0,\infty)$ and  $C^2((0,\infty) \backslash \{ c_2^*\}) $) when $X$ is of bounded (resp.\ unbounded) variation.   

Let $\mathcal{L}$ be the infinitesimal generator associated with
the process $X$ applied to a sufficiently smooth function $f$
\begin{align*}
\mathcal{L} f(x) &:= -c f'(x) + \frac 1 2 \sigma^2 f''(x) + \int_0^\infty \left[ f(x+z) - f(x) -  f'(x) z 1_{\{0 < z < 1\}} \right] \nu(\diff z).
\end{align*}
Here $\mathcal{L} v_{c_1^*,c_2^*} (\cdot)$ makes sense anywhere on $(0,\infty) \backslash \{ c_2^*\}$.  

\begin{proposition} \label{verification_1} 
\begin{enumerate}
\item $(\mathcal{L}-q) v_{c_1^*,c_2^*}(x) = 0$ for $0 < x < c_2^*$,
\item $(\mathcal{L}-q) v_{c_1^*,c_2^*}(x) \leq 0$ for $x > c_2^*$.
\end{enumerate}
\end{proposition}
\begin{proof}
(1) By Proposition 2 of \cite{Avram_et_al_2007} and as in the proof of Theorem 8.10 of \cite{Kyprianou_2006}, the processes 
\begin{align*}
e^{-q (t \wedge \tau^-_0 \wedge \tau^+_{c_2^*})} Z^{(q)}(X_{t \wedge \tau^-_0 \wedge \tau^+_{c_2^*}}) \quad \textrm{and} \quad e^{-q (t \wedge \tau^-_0 \wedge \tau^+_{c_2^*})} R^{(q)}(X_{t \wedge \tau^-_0 \wedge \tau^+_{c_2^*}}), \quad t \geq 0,
\end{align*}
are martingales. Thanks to the smoothness of $Z^{(q)}$ and $R^{(q)}$ on $(0,c_2^*)$ (see Remark \ref{remark_smoothness_zero}(1)), we obtain $(\mathcal{L}-q) R^{(q)}(y) = (\mathcal{L}-q) Z^{(q)}(y) = 0$ for any $0 < y < c_2^*$. This step is similar to the proof of  Theorem 2.1 in \cite{Bayraktar_2012}.  This implies claim (1) in view of \eqref{value_function}.

\noindent (2)  Suppose $X$ is of bounded variation.  By \eqref{v_prime_limit} and  Remarks \ref{remark_smoothness_zero}(2) and \ref{remark_gamma_positive}, 
\begin{align*} 
v_{c_1^*,c_2^*}' (c_2^*-) < 1 = v_{c_1^*,c_2^*}' (c_2^*+). 
\end{align*}
Because $\sigma = 0$, $\mathfrak{d} > 0$ and $v_{c_1^*,c_2^*}(\cdot)$ is continuous at $c_2^*$, (1) implies $(\mathcal{L}-q) v_{c_1^*,c_2^*}(c_2^*+) < 0$.  Because, on $(c_2^*,\infty)$, $\mathcal{L} v_{c_1^*,c_2^*}$ is a constant and $qv_{c_1^*,c_2^*}$ is increasing in view of \eqref{value_function}, claim (2) follows for the bounded variation case.

Suppose $X$ is of unbounded variation.  By \eqref{v_double_prime_limit} and  Remarks \ref{remark_smoothness_zero}(2) and \ref{remark_gamma_positive},
\begin{align*} 
v_{c_1^*,c_2^*}'' (c_2^*-) 
&=  \gamma (c_1^*,c_2^*) q W^{(q)'}(0+) > 0  = v_{c_1^*,c_2^*}'' (c_2^*+).
\end{align*}
Because $v_{c_1^*,c_2^*}$ is differentiable at $c_2^*$, we must have $(\mathcal{L}-q) v_{c_1^*,c_2^*}(c_2^*+) < 0$.  Again, because $(\mathcal{L}-q)v_{c_1^*,c_2^*}$ is decreasing on $(c_2^*,\infty)$, (2) is proved for the unbounded variation case as well.
\end{proof}

\begin{proposition} \label{verification_2} 
For any $x > y \geq 0$, it holds that $v_{c_1^*,c_2^*} (x) -  v_{c_1^*,c_2^*}(y) 
 \geq x-y - \beta$.
\end{proposition}
In order to show this proposition, we take advantage of the slope of $v_{c_1^*,c_2^*}$ at $c_1^*$.
By \eqref{v_prime}, 
\begin{align*}
v_{c_1^*,c_2^*}' (c_1^*) 
=  Z^{(q)}(c_2^*-c_1^*)  -  \gamma (c_1^*,c_2^*) q W^{(q)}(c_2^*-c_1^*) =  1- H(c_1^*,c_2^*).
\end{align*}
When $c_1^* = 0$, the derivative is understood as the right-derivative.
Hence we arrive at the following.
%

\begin{lemma}[slope at $c_1^*$]  \label{lemma_v_derivative_c_1}
For both \textbf{Cases 1} and \textbf{2}, $v_{c_1^*,c_2^*}'(c_1^*+)\leq 1$.  In particular, for \textbf{Case 1}, $v_{c_1^*,c_2^*}'(c_1^*)=1$.
\end{lemma}

\begin{lemma} \label{lemma_derivative_bounded}
For any $x \in (0,\infty) \backslash \{c_2^*\}$, $v_{c_1^*,c_2^*}'(x) < 1$ if and only if $x \in (c_1^*,c_2^*)$. 
\end{lemma}
\begin{proof}
Because $v_{c_1^*,c_2^*}' (x) = 1$ on $(c_2^*, \infty)$, we shall focus on $x \in (0,c_2^*)$.  Rewriting \eqref{v_double_prime}, 
\begin{align} 
v_{c_1^*,c_2^*}''(x)   
&= -q  W^{(q)}(c_2^*-x) J(x), \quad 0 < x < c_2^*, \label{second_derivative_of_v}
\end{align}
where $J(x) :=  1 -  \gamma(c_1^*,c_2^*) \frac {W^{(q)'}(c_2^*-x)} { W^{(q)}(c_2^*-x)}$.  By Remarks \ref{remark_smoothness_zero}(3) and  \ref{remark_gamma_positive}, $J(\cdot)$ is decreasing on $(0,c_2^*)$, and hence there exists a unique level $\bar{c} \in [0,c_2^*]$ such that \eqref{second_derivative_of_v} is negative if and only if $x < \bar{c}$.   In other words, there are three possibilities 
\begin{enumerate}
\item[(i)] $v_{c_1^*,c_2^*}$ is strictly concave on $(0,c_2^*)$ or
\item[(ii)] $v_{c_1^*,c_2^*}$ is strictly concave on  $(0, \bar{c})$ and strictly convex on  $(\bar{c},c_2^*)$,
\item[(iii)] $v_{c_1^*,c_2^*}$ is strictly convex on $(0,c_2^*)$.
\end{enumerate}

\textbf{Case 1}:  Suppose $c_1^* > 0$ with $H(c_1^*,c_2^*) = 0$. 
By Lemma \ref{lemma_v_derivative_c_1}, \eqref{v_prime_limit} and Remark  \ref{remark_gamma_positive},
\begin{align} 
v_{c_1^*,c_2^*}' (c_2^*-)
 \leq 1 = v_{c_1^*,c_2^*}'(c_1^*). \label{v_prime_relation}
\end{align}
Therefore we can safely rule out (iii) and we must have either (i) or (ii)  with $c_1^* < \bar{c} < c_2^*$.  For (i) (thus  $v_{c_1^*,c_2^*}'$ is decreasing on $(0,c_2^*)$), given $x \in (0,c_2^*)$, $v_{c_1^*,c_2^*}'(x) < 1$ if and only if $x \in (c_1^*,c_2^*)$.  Now suppose (ii) with  $c_1^* < \bar{c} < c_2^*$.  Then by the concavity on $(0,\bar{c})$ and $1 = v_{c_1^*,c_2^*}'(c_1^*)$, we have $v_{c_1^*,c_2^*}' > 1$ on $(0,c_1^*)$ and $v_{c_1^*,c_2^*}' < 1$ on $(c_1^*, \bar{c})$.  For $x \in (\bar{c},c_2^*)$, by the convexity on $(\bar{c}, c_2^*)$ and \eqref{v_prime_relation}, $1 \geq v_{c_1^*,c_2^*}' (c_2^*-) \geq v_{c_1^*,c_2^*}' (x)$.

\textbf{Case 2}: Suppose $c_1^* = 0$ with $H(0,c_2^*) \geq 0$.  
In view of \eqref{v_prime} and the definition of $H(0,c_2^*)$, we must have that $v'_{0,c_2^*}(0+) \leq 1$.  This together with $v'_{0,c_2^*}(c_2^*-) \leq 1$ shows that  $v'_{0,c_2^*}(x) < 1$ on $(0,c_2^*)$ for any of (i), (ii) and (iii).
\end{proof}

By Lemma \ref{lemma_derivative_bounded}, 
\begin{align*}
\inf_{x > y} [v_{c_1^*,c_2^*} (x) -  v_{c_1^*,c_2^*}(y) 
- (x-y - \beta)] = v_{c_1^*,c_2^*} (c_2^*) -  v_{c_1^*,c_2^*}(c_1^*) 
- (c_2^*-c_1^* - \beta) = 0,
\end{align*}
and as a result the claim in Proposition \ref{verification_2} follows immediately.

Next, we will verify the optimality of the $(c_1^*,c_2^*)$-policy.

\begin{theorem} \label{theorem_main}We have  $v_{c_1^*,c_2^*}(x)=\sup_{\pi\in \Pi}v_\pi(x)$ for every $x \geq 0$ and the $(c_1^*,c_2^*)$-policy is optimal.
\end{theorem}
\begin{proof}
Here we only provide a sketch of a proof since it is similar to that of  Lemma 6 of \cite{Loeffen_2009_2}.
To verify the optimality of  $(c_1^*,c_2^*)$ we only need to show that $v_{c_1^*,c_2^*}(x) \geq v_{\pi}(x)$, $x \geq 0$, for all $\pi \in \Pi$. But this result follows from applying the It\^{o} formula to $v_{c_1^*,c_2^*}(U_t^{\pi})$ for an arbitrary $\pi \in \Pi$, using  Propositions \ref{verification_1} and  \ref{verification_2} and then passing to the limit using Fatou's lemma. Here one should be careful in applying the It\^{o} formula since the value function $v_{c_1^*,c_2^*}$ may not be smooth enough at $c_2^*$ to apply the usual version. When $X$ of unbounded variation, we use Theorem 3.2 of \cite{MR2408999}, which shows that the smooth fit principle (which we proved in Lemma~\ref{lemma_smoothness_c_2})  is enough to kill the local time terms that might accumulate around $c_2^*$; see also Theorem IV.71 of \cite{MR2273672}, or Exercise 3.6.24 of \cite{MR1121940}. On the other hand, when $X$ is of bounded variation recall from Lemma~\ref{lemma_smoothness_c_2} that the value function is only continuous. However, in this case we do not  need the smoothness of the value function at $c_2^*$, simply because the first derivative term is integrated against the Lebesgue measure which is a diffuse measure. We could also directly use the first part of  Theorem 6.2 of \cite{Oksendal_Sulem_2007}. 
\end{proof}

We conclude this section by showing the uniqueness of $(c_1^*, c_2^*)$; recall that the existence was proved in Proposition \ref{proposition_existence}.
\begin{proposition}
The maximizer $(c_1^*, c_2^*)$ is unique.
\end{proposition}
\begin{proof}
Suppose $(c_1^*, c_2^*)$ and $(\hat{c}_1^*, \hat{c}_2^*)$ both maximize $\bar{v}_{c_1,c_2}-c_1$.  We shall show that they must be equal. 

By Lemma \ref{lemma_h_g}, both $(c_1^*, c_2^*)$ and $(\hat{c}_1^*, \hat{c}_2^*)$ satisfy \textbf{Case 1} or  \textbf{Case 2} and by Theorem \ref{theorem_main} we have
\begin{align}
v_{c_1^*,c_2^*}(x) = v_{\hat{c}_1^*,\hat{c}_2^*}(x) =\sup_{\pi\in \Pi}v_\pi(x) \quad x \geq 0. \label{uniqueness_equal}
\end{align}
We first show that $c_2^* =  \hat{c}_2^*$.  Indeed, by Lemma \ref{lemma_derivative_bounded}, $v_{c_1^*,c_2^*}'(x) < 1$ on $(c_1^*,c_2^*)$ and $v_{\hat{c}_1^*,\hat{c}_2^*}'(x) < 1$ on $(\hat{c}_1^*,\hat{c}_2^*)$ while $v_{c_1^*,c_2^*}'(x) = 1$ on $(c_2^*,\infty)$ and $v_{\hat{c}_1^*,\hat{c}_2^*}'(x) = 1$ on $(\hat{c}_2^*,\infty)$.  Hence if $c_2^*  \neq  \hat{c}_2^*$, it would contradict with \eqref{uniqueness_equal} for the points between $c_2^*$ and $\hat{c}_2^*$.

In order to show $c_1^* =  \hat{c}_1^*$, we appeal to the identity $v_{c_1,c_2} (c_2) -  v_{c_1,c_2}(c_1) 
= c_2-c_1 - \beta$, which holds for any $0 \leq c_1 < c_2$ for which $v_{c_1,c_2}$ is continuous at $c_2$.  This together with \eqref{uniqueness_equal} and  $c_2^* =  \hat{c}_2^*$ shows $v_{c_1^*,c_2^*}(\hat{c}_1^*) -  v_{c_1^*,c_2^*}(c_1^*) 
= \hat{c}_1^* -c_1^*$.  If $\hat{c}_1^* \neq c_1^*$, by the mean value theorem, there exists a point between these at which $v_{c_1^*,c_2^*}'$ is one; however, this contradicts Lemma \ref{lemma_derivative_bounded}.  This completes the proof.
\end{proof}

\section{Numerical Examples} \label{section_numerics}
In this section, we confirm the results numerically using the spectrally positive \lev process with i.i.d.\ phase-type distributed jumps \cite{Asmussen_2004} of the form
\begin{equation*}
  X_t  - X_0= -\mathfrak{d} t+\sigma B_t + \sum_{n=1}^{N_t} Z_n, \quad 0\le t <\infty, 
\end{equation*}
for some $\mathfrak{d} \in \R$ and $\sigma \geq 0$.  Here $B=\{B_t; t\ge 0\}$ is a standard Brownian motion, $N=\{N_t; t\ge 0\}$ is a Poisson process with arrival rate $\lambda$, and  $Z = \left\{ Z_n; n = 1,2,\ldots \right\}$ is an i.i.d.\ sequence of phase-type-distributed random variables with representation $(m,{\bm \alpha},{\bm T})$; see \cite{Asmussen_2004}.
The processes $N$, $B$ and $Z$ are assumed to be mutually independent. Its Laplace exponent \eqref{laplace_spectrally_positive} is then
\begin{align*}
 \psi(s)   = \mathfrak{d} s + \frac 1 2 \sigma^2 s^2 + \lambda \left( {\bm \alpha} (s {\bm I} - {\bm{T}})^{-1} {\bm t} -1 \right),
 \end{align*}
which is analytic for every $s \in \mathbb{C}$ except at the eigenvalues of ${\bm T}$.  Suppose $\{ -\xi_{i,q}; i \in \mathcal{I}_q \}$ is the set of the roots of the equality $\psi(s) = q$ with negative real parts, and if these are assumed distinct, then
the scale function can be written
\begin{align}
W^{(q)}(x) =  \frac {e^{\Phi(q) x}} {\psi'(\Phi(q))}  - \sum_{i \in \mathcal{I}_q}  C_{i,q} e^{-\xi_{i,q} x}, \quad x \geq 0, \label{scale_function_distinct}
\end{align}
where
\begin{align*}
C_{i,q} &:= \left. \frac { s+\xi_{i,q}} {q-\psi(s)} \right|_{s = -\xi_{i,q}} = - \frac 1 {\psi'(-\xi_{i,q})};
\end{align*}
see \cite{Egami_Yamazaki_2010_2}.  Here $\{ \xi_{i,q}; i \in \mathcal{I}_q \}$ and  $\{ C_{i,q}; i \in \mathcal{I}_q \}$ are possibly complex-valued.

In our example, we shall choose a phase-type distribution which does not have a completely monotone density.
Recall that, in the spectrally negative counterpart \cite{Loeffen_2009_2}, the $(c_1,c_2)$-policy may fail to be optimal if the \lev density is not completely monotone.  On the other hand, in the dual model, there is no restriction on the \lev measure. We assume $m = 6$ and 
\begin{align*}
&{\bm T} = \left[ \begin{array}{rrrrrr}   -5.6546  &  0.0000   &      0.0000  &  0.0000  &  0.0000  &  0.0000 \\
    0.6066  & -5.6847 &   0.0000  &  0.0166  &  0.0089  &  5.0526 \\
    0.2156  &  4.3616  & -5.6485  &  0.9162 &   0.1424 &   0.0126 \\
    5.6247 &   0.0000   & 0.0000  & -5.6786  &  0.0000  &  0.0000 \\
    0.0107  &  0.0000  &  0.0000 &   5.7247 &  -5.7420   & 0.0000 \\
    0.0136  &  0.0000 &   0.0000  &  0.0024 &   5.7022 &  -5.7183 \end{array} \right],  \quad {\bm \alpha} =    \left[ \begin{array}{l}   
    0.0000 \\
    0.0007 \\
    0.9961 \\
    0.0000 \\
    0.0001 \\
    0.0031  \end{array} \right],
\end{align*}
which give an approximation of the Weibull distribution with density function $f(x) = \alpha \gamma^\alpha x^{\alpha -1} \exp \left\{ - (\gamma x)^\alpha\right\}$ for $\alpha = 2$ and $\gamma = 1$, obtained using the EM-algorithm; see \cite{Egami_Yamazaki_2010_2} regarding the approximation performance of the corresponding scale function.  Throughout this section, we let $q = 0.05$ and let other parameters vary so as to see their impacts on the optimal strategy and the value function.

In our first experiment, we let $\mathfrak{d} = 2$, $\sigma = 0$ or $\sigma = 1$ with
\begin{description}
\item[Case 1] $\beta = 4$ and $\lambda=3$
\item[Case 2] $\beta = 4$ and $\lambda=1$
\end{description}
and obtain the optimal strategies/value functions and confirm the analytical results obtained in the previous sections.  We choose these parameters so that $c_1^*> 0$ for Case 1 and $c_1^*= 0$ for Case 2.

Figures \ref{figure_with_diffusion} and \ref{figure_no_diffusion} show the results for $\sigma = 1$ and $\sigma =0$, respectively.  In both figures, we plot in the left column $\bar{v}_{c_1,c_2}-c_1$ with respect to $c_1$ and $c_2$ and in the right column the value function $v_{c_1^*,c_2^*}(\cdot)$ as a function of the initial value $x$.  Recall that the values $(c_1^*,c_2^*)$ are those that maximize $\bar{v}_{c_1,c_2}-c_1$.  As can be suggested from the contour map of $\bar{v}_{c_1,c_2}-c_1$, there exists a unique global maximum and hence Newton's method is a reasonable choice of computing the maximizer $(c_1^*,c_2^*)$. 
For the plots of the value functions, the circles indicate the points $(c_1^*, v_{c_1^*,c_2^*}(c_1^*))$ and $(c_2^*, v_{c_1^*,c_2^*}(c_2^*))$ and  the dotted lines the 45-degree lines passing through these points.

In view of these figures,  the continuity/smoothness at $c_2^*$ is readily confirmed; it appears to be differentiable for the case $\sigma = 1$ (in other words, the value function is tangent to the 45-degree line) while it is continuous for the case $\sigma = 0$.  The non-differentiability for $\sigma = 0$ is apparent in view of Case 2 in Figure \ref{figure_no_diffusion}.  At $c_1^*$, the value function is indeed tangent to the 45-degree line  if $c_1^* > 0$, while for the case $c_1^* = 0$, we see that the slope is less than one.  These results are consistent with Proposition \ref{proposition_existence}.  It is also confirmed that the slope of $v_{c_1^*,c_2^*}$ is smaller than $1$ only at those points inside $[c_1^*,c_2^*]$, which verifies Lemma \ref{lemma_derivative_bounded} and Proposition \ref{verification_2}.   

\begin{figure}[htbp]
\begin{center}
\begin{minipage}{1.0\textwidth}
\centering
\begin{tabular}{cc}
\includegraphics[scale=0.58]{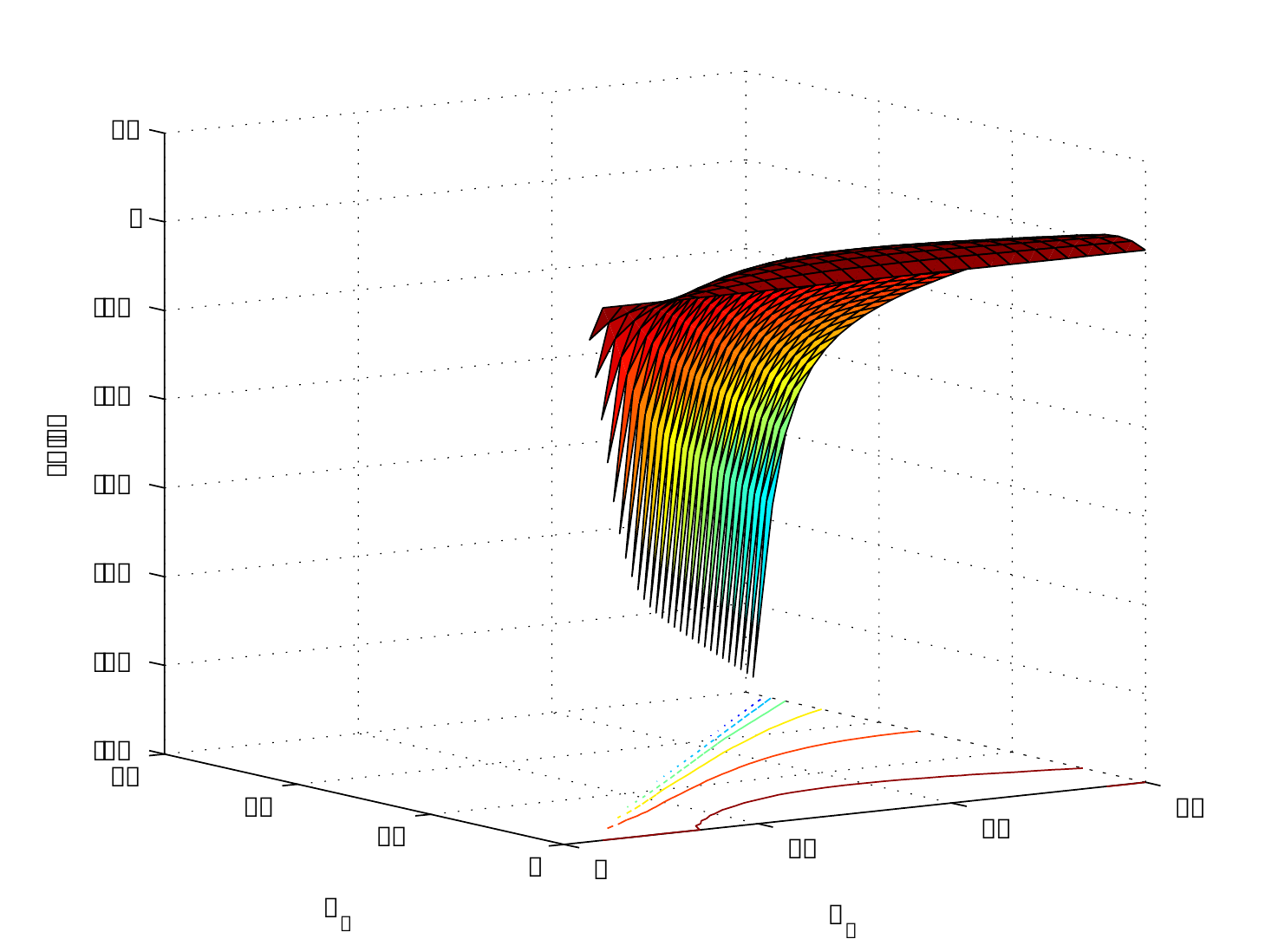}  & \includegraphics[scale=0.58]{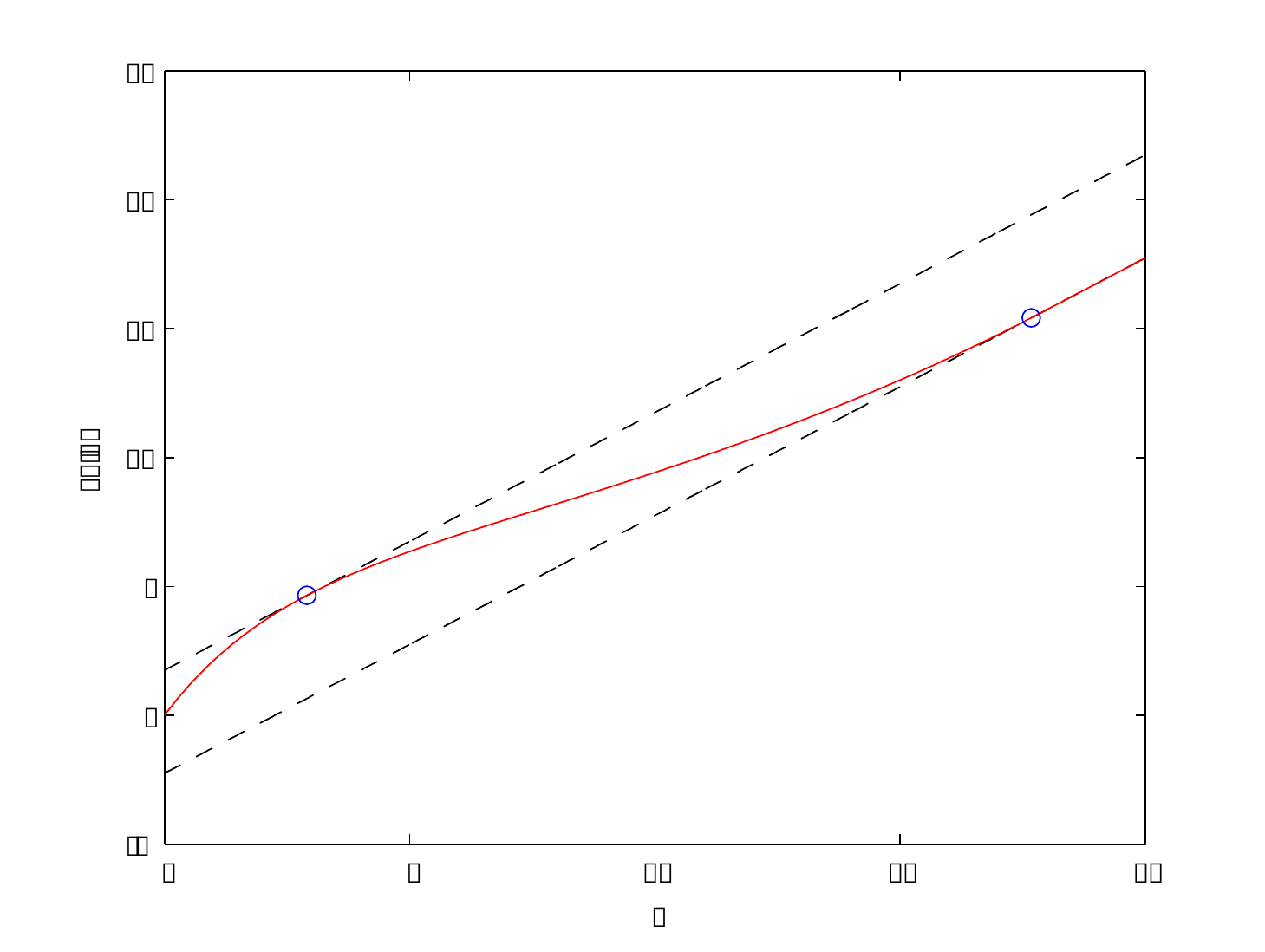} \\
\multicolumn{2}{c}{Case 1: $\beta = 4$ and $\lambda=3$}\vspace{0.3cm} \\
\includegraphics[scale=0.58]{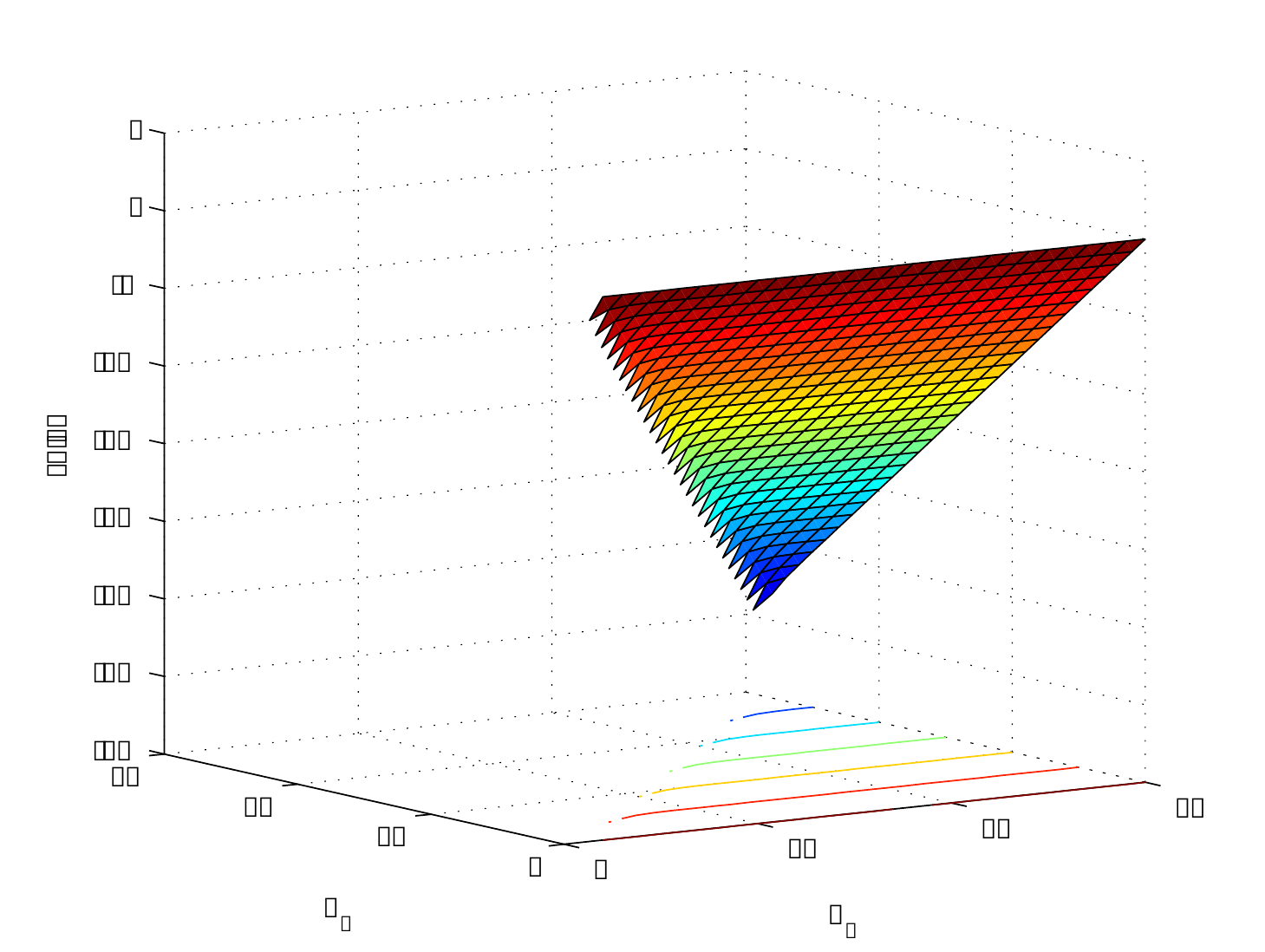}  & \includegraphics[scale=0.58]{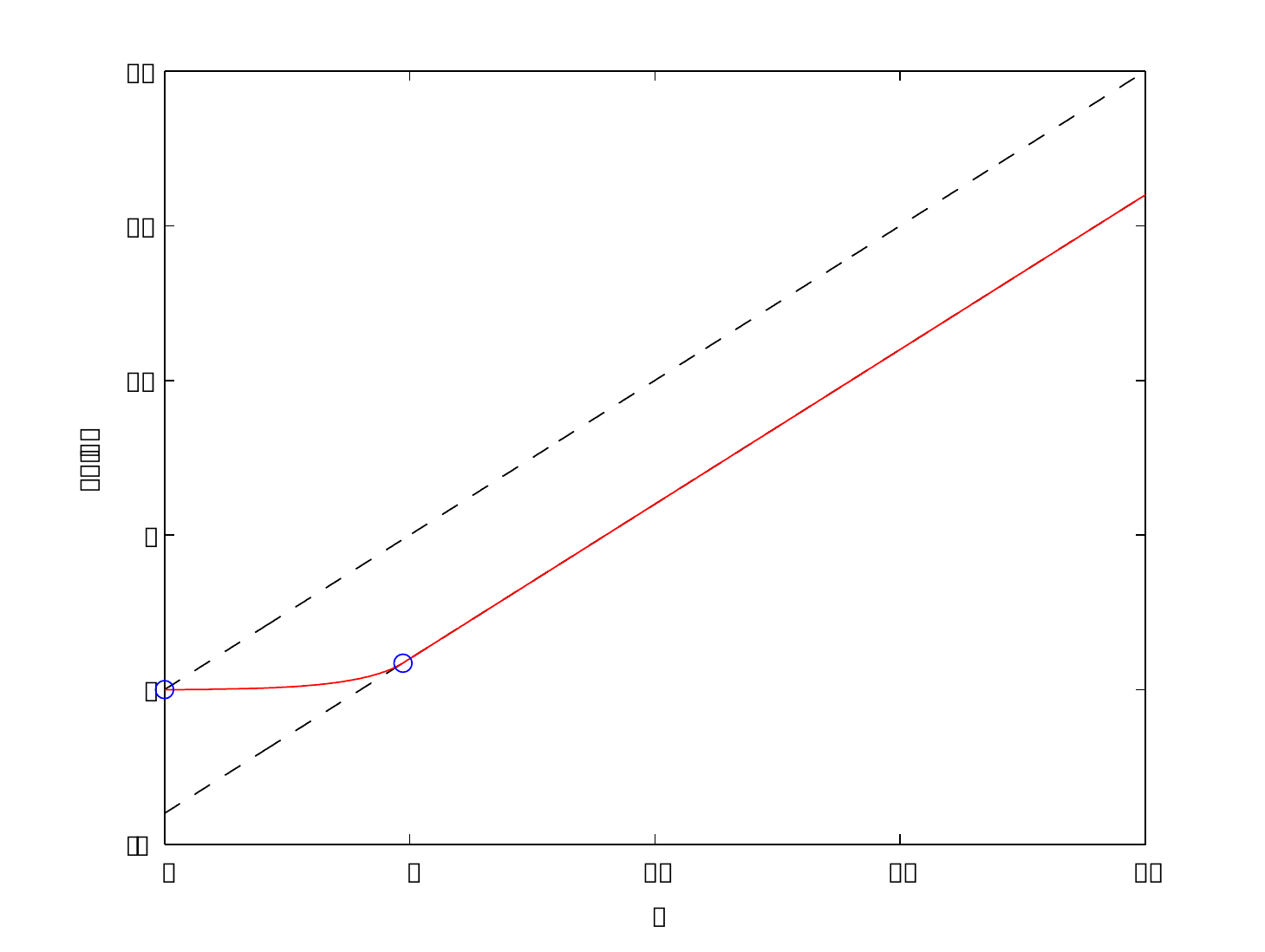} \\
\multicolumn{2}{c}{Case 2: $\beta = 4$ and $\lambda=1$}\vspace{0.3cm} \\
\end{tabular}
\end{minipage}
\caption{For the case $\sigma = 1$: (left) $\bar{v}_{c_1,c_2}-c_1$ with respect to $c_1$ and $c_2$, (right) the value function $v_{c_1^*,c_2^*}$ as a function of $x$.} \label{figure_with_diffusion}
\end{center}
\end{figure}

\begin{figure}[htbp]
\begin{center}
\begin{minipage}{1.0\textwidth}
\centering
\begin{tabular}{cc}
\includegraphics[scale=0.58]{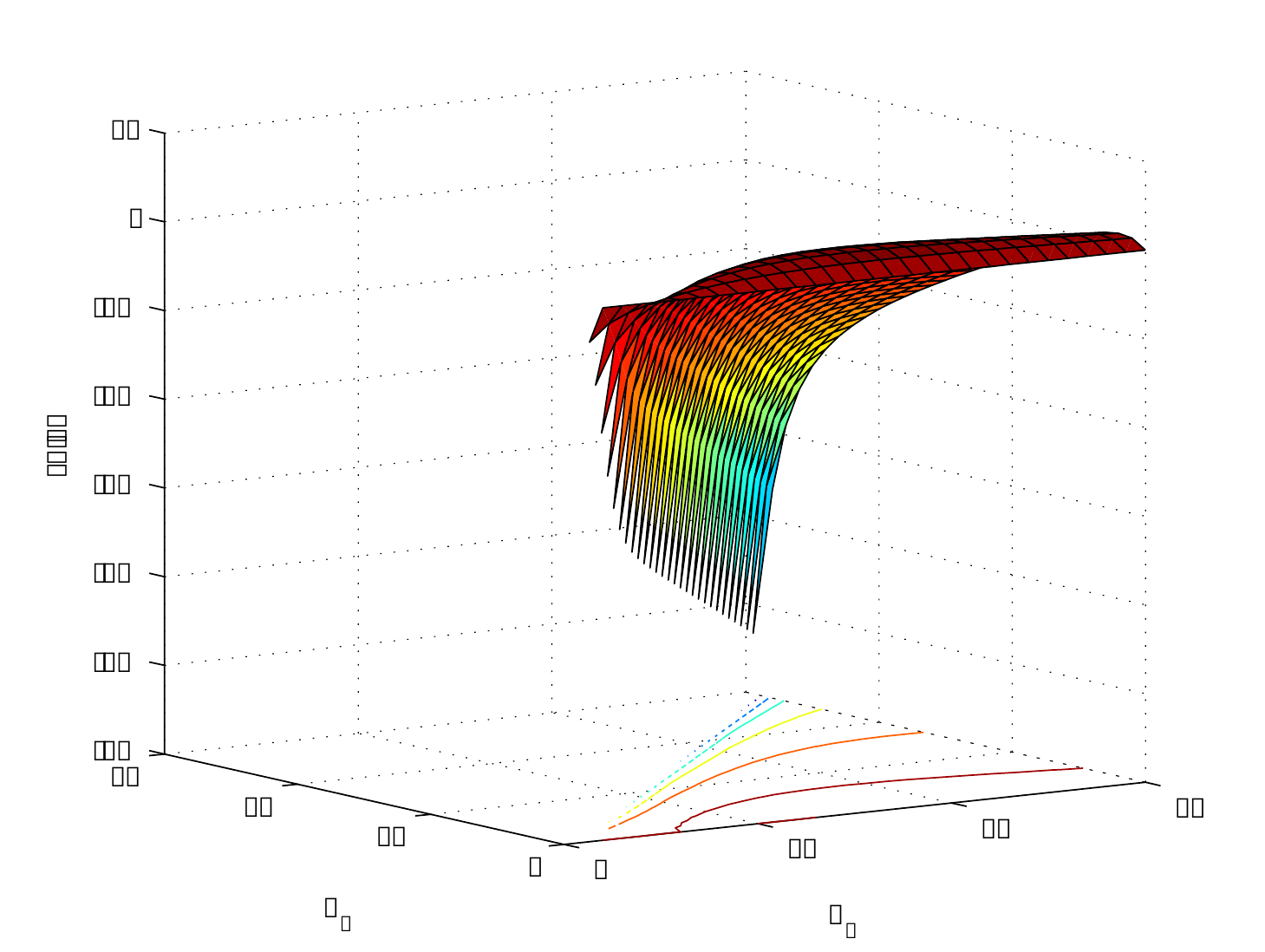}  & \includegraphics[scale=0.58]{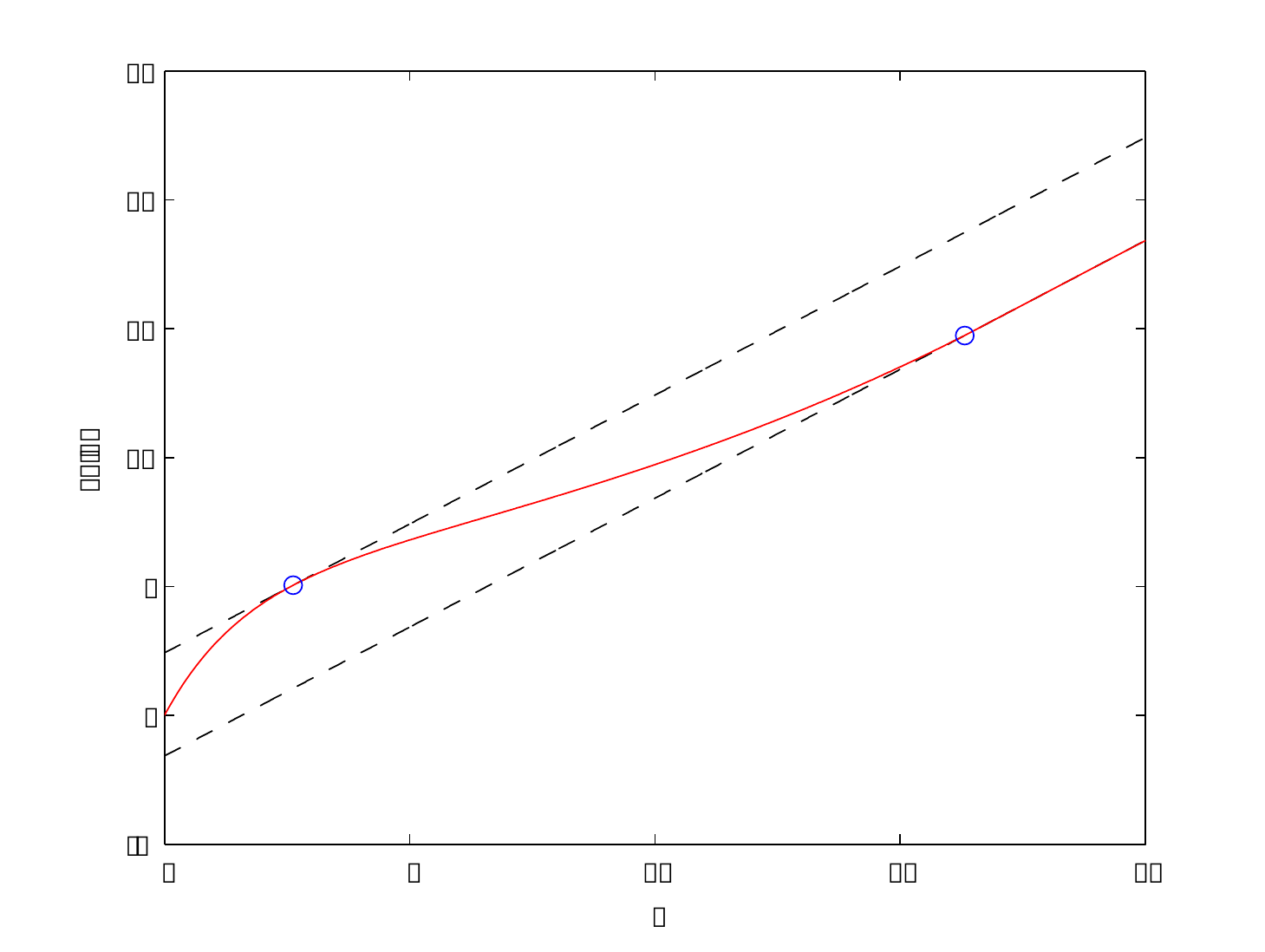} \\
\multicolumn{2}{c}{Case 1: $\beta = 4$ and $\lambda=3$}\vspace{0.3cm} \\
\includegraphics[scale=0.58]{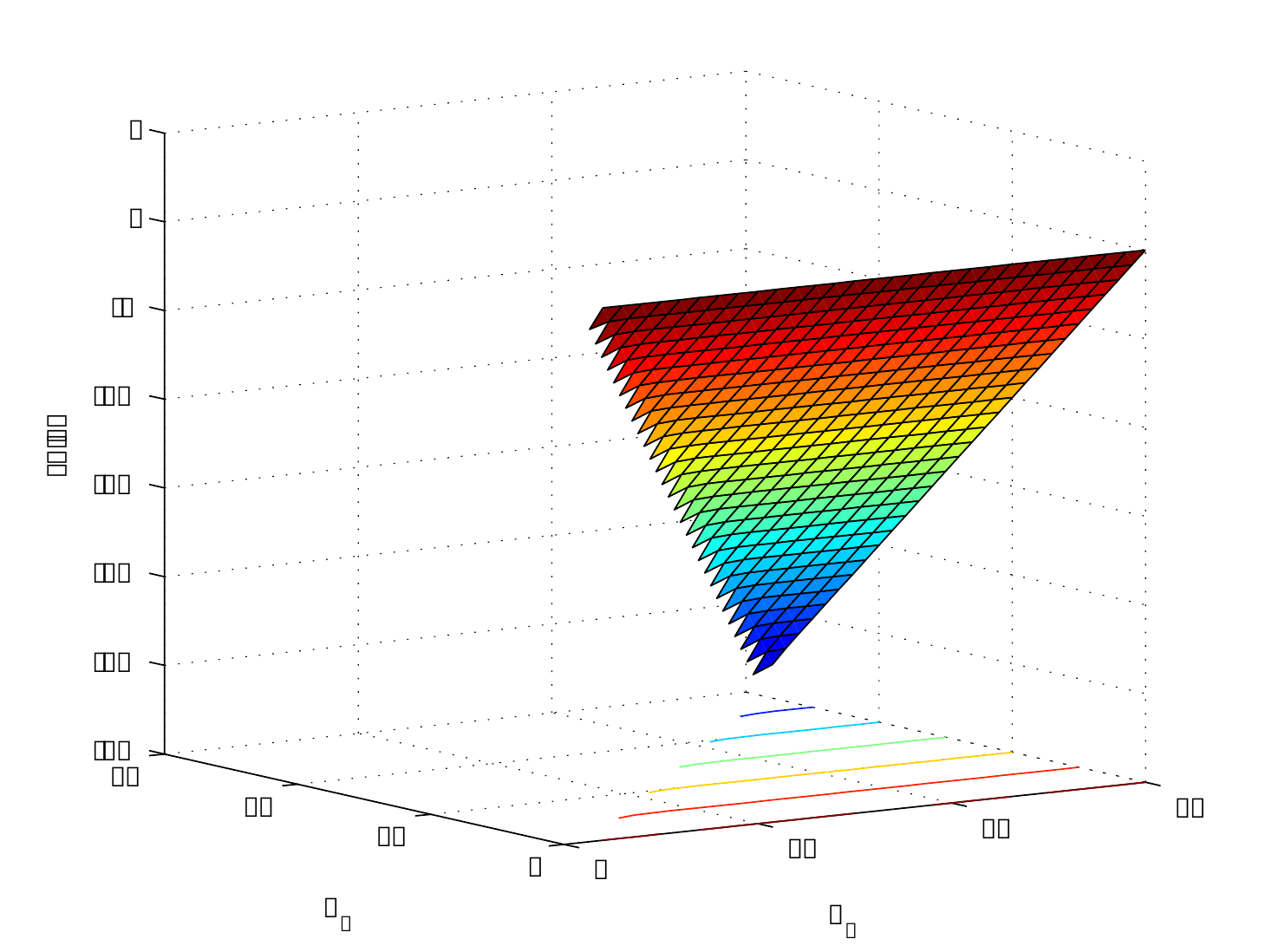}  & \includegraphics[scale=0.58]{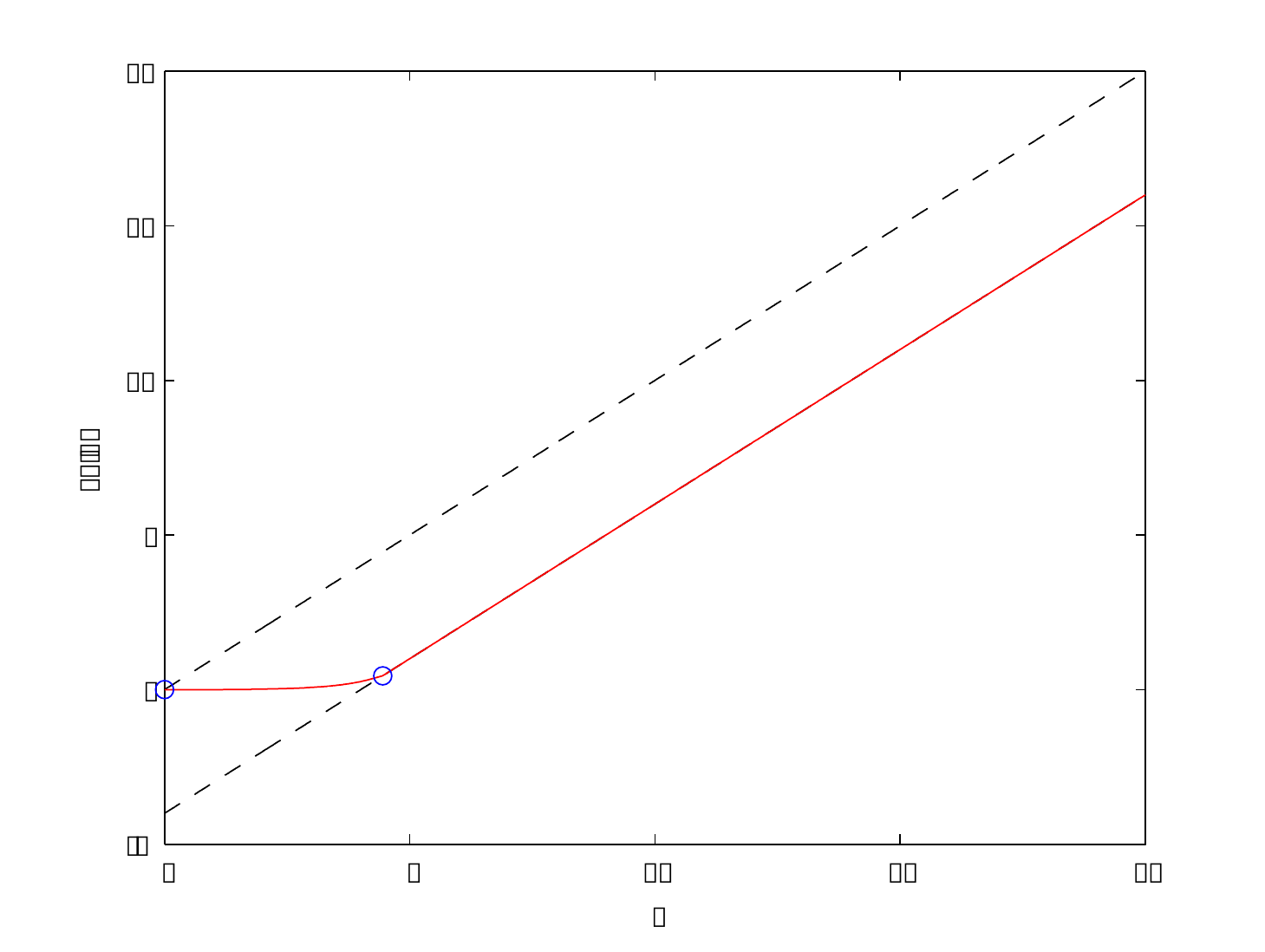} \\
\multicolumn{2}{c}{Case 2: $\beta = 4$ and $\lambda=1$}\vspace{0.3cm} \\
\end{tabular}
\end{minipage}
\caption{For the case with $\sigma = 0$: (left) $\bar{v}_{c_1,c_2}-c_1$ with respect to $c_1$ and $c_2$, (right) the value function $v_{c_1^*,c_2^*}$ as a function of $x$.} \label{figure_no_diffusion}
\end{center}
\end{figure}

In our second experiment, we take $\beta \downarrow 0$ and see if the value function converges to the one under no-transaction costs as in \cite{Bayraktar_2012}:
\begin{equation} \label{value_function_no_transaction}
\hat{v}_{a^*}(x) :=
\begin{cases}
 - \overline{R}^{(q)} (a^*-x), & \text{if $\mu>0$}, 
 \\ x, & \text{if $\mu \leq 0$},
 \end{cases}  
 \end{equation}
for any $x \geq 0$, with the optimal barrier level 
\begin{equation*}
a^* := \begin{cases} \left(\overline{Z}^{(q)}\right)^{-1}\left (\frac \mu q \right)> 0 & \text{if $\mu>0$}, \\
0   & \text{if $\mu\leq 0$}. \end{cases}
\end{equation*}
We let $\lambda = 3$ and consider the case $\mu > 0$ (by choosing $\mathfrak{d} = 2$) and also the case $\mu < 0$ (by choosing $\mathfrak{d} = 3$).

Figure \ref{figure_beta} plots for each case the value function $v_{c_1^*,c_2^*}(\cdot)$ for $\beta = 10,5,1,0.5,0.1$ (dotted) together with the no-transaction case $\hat{v}_{a^*}(\cdot)$ (solid) as in \eqref{value_function_no_transaction}.  The circles on the plots indicate the points $(c_1^*, v_{c_1^*,c_2^*}(c_1^*)), (c_2^*, v_{c_1^*,c_2^*}(c_2^*))$ and also $(a^*, \hat{v}_{a^*}(a^*))$.  It is easy to see that the value function is monotone in $\beta$ (uniformly in $x$), and converge to the no-transaction cost case as $\beta \downarrow 0$.  The convergences of both $c_1^*$ and $c_2^*$  to $a^*$ are also observed. In fact, one can prove the convergence of value functions using the stability of viscosity solutions.
\begin{proposition}
Let $v^{\beta}$ denote the value function corresponding to the dividend payment problem when the fixed transaction cost is $\beta$ (defined as above), and $\hat{v}$ the value function when there are no-transaction costs. Then $v^{\beta}$ converges to $\hat{v}$ uniformly as $\beta \downarrow 0$.
\end{proposition}

\begin{proof}
From the definition of the problem, $v^{\beta} \leq \hat{v}$ and $v^{\beta}$ is decreasing in $\beta$ and hence it has a point-wise limit, which we will call $\tilde{v}$. The proof is completed if we can show that $\tilde{v}$ is a viscosity super-solution of the variational inequality that corresponds to the problem without transaction costs. But this is an immediate consequence of the stability result of the viscosity solutions (see e.g.\ Theorem 6.8 of \cite{MR2976505} and Theorem 1 of \cite{Barles_2008}), since we can obtain the variational inequality in the no-transaction case by taking a limit in the case with transaction costs.

To get to uniform convergence from point-wise convergence we just proved, we appeal to Dini's theorem to first show it on compacts. This indeed holds because we already know that $(v^{\beta})$ and $\hat{v}$ are continuous functions and $v^{\beta} \uparrow \hat{v}$ as $\beta \downarrow 0$.  Now, because the slopes of $(v^{\beta})$ and $\hat{v}$ are all one above $c_2^*$ and $a^*$, respectively, and because $c_2^*$ can be shown to be bounded for any small $\beta$ (thanks to the convergence $c_2^*$ to $a^*$ as $\beta \downarrow 0$ or modifying the proof of Lemma \ref{lemma_sup_bounded_set}), the uniform convergence holds.
\end{proof}

We also observe in the figures that for  $\mu < 0$, $c_1^* = 0$.  This can be shown analytically for any $\beta > 0$.  
\begin{corollary}
If $\mu \leq 0$, we must have $c_1^* = 0$.
\end{corollary}

\begin{proof}
By the nature of the problem the value function $v_{c_1^*,c_2^*}$ is dominated by that of the no-transaction cost case.  By \eqref{value_function_no_transaction}, we must have $v_{c_1^*,c_2^*}(x) \leq x$ for any $x \geq 0$.  Moreover, because $v_{c_1^*,c_2^*}(0) = 0$, $v_{c_1^*,c_2^*}'(0+) < 1$ and hence, in view of the proof of Proposition \ref{lemma_derivative_bounded}, we must have $c_1^* = 0$.
\end{proof}

\begin{figure}[htbp]
\begin{center}
\begin{minipage}{1.0\textwidth}
\centering
\begin{tabular}{cc}
\includegraphics[scale=0.58]{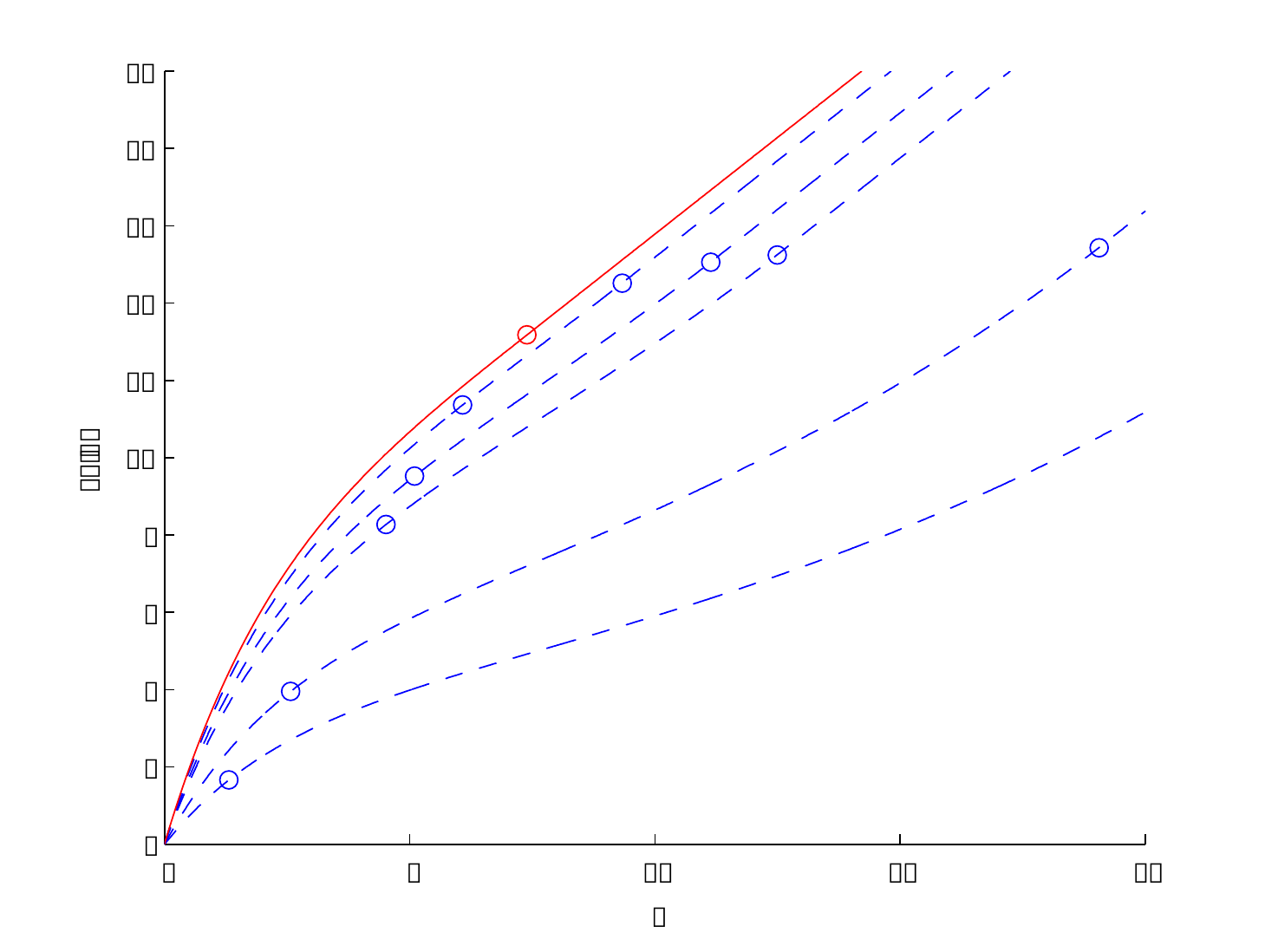}  & \includegraphics[scale=0.58]{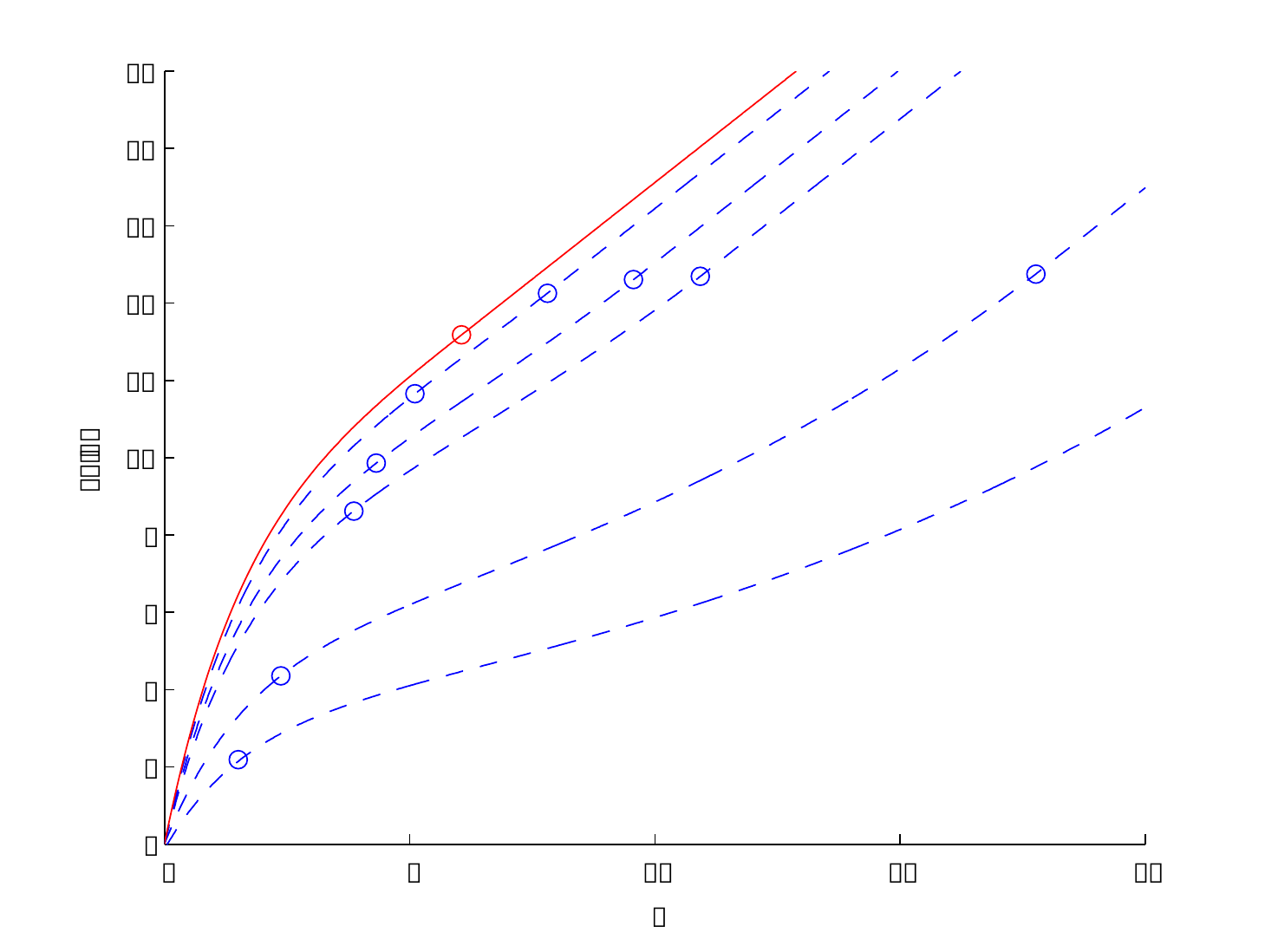} \\
\multicolumn{2}{c}{When $\mu > 0$: (left) $\sigma =1 $ and (right) $\sigma =0$} \vspace{0.3cm} \\
\includegraphics[scale=0.58]{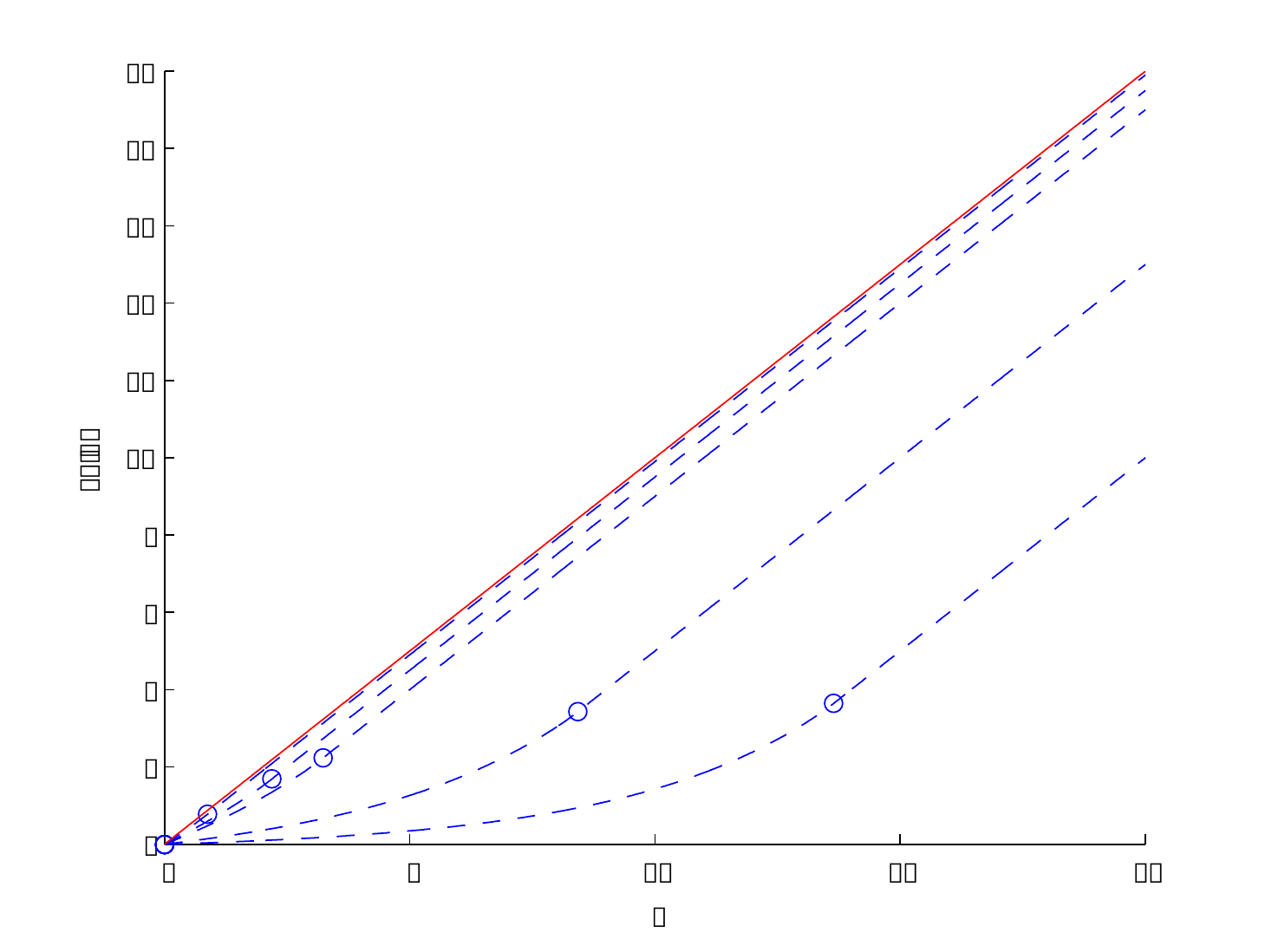}  & \includegraphics[scale=0.58]{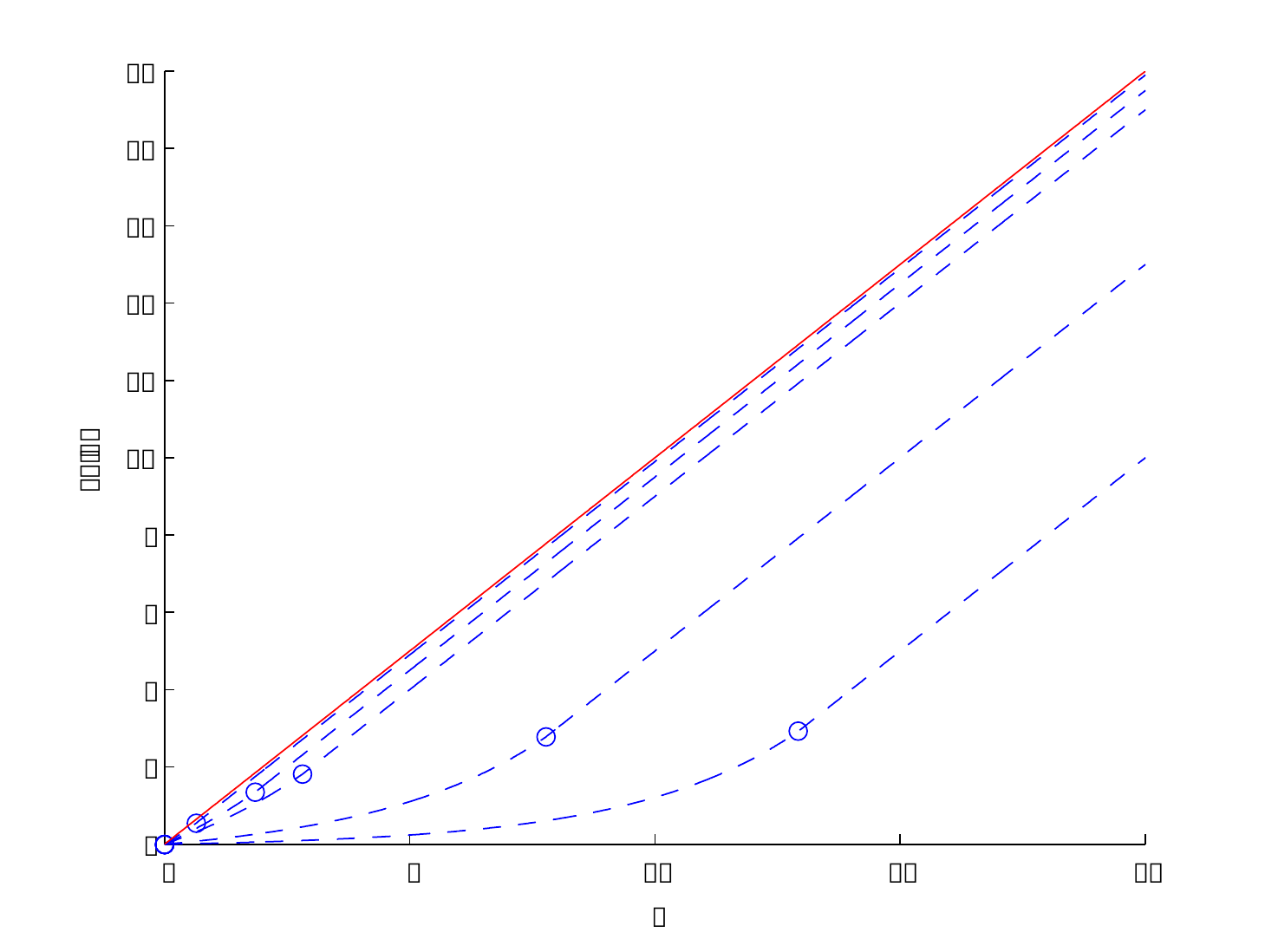} \\
\multicolumn{2}{c}{When $\mu < 0$: (left) $\sigma =1 $ and (right) $\sigma =0$} \vspace{0.3cm} \\
\end{tabular}
\end{minipage}
\caption{Convergence as $\beta \downarrow 0$.} \label{figure_beta}
\end{center}
\end{figure}

\section*{Acknowledgements}
E. Bayraktar is supported in part by the National Science Foundation under a Career grant DMS-0955463 and in part by the Susan M. Smith Professorship. A. Kyprianou would like to thank FIM (Forschungsinstitut f\"ur Mathematik) for supporting him during his sabbatical at ETH, Zurich.
K.\ Yamazaki is in part supported by Grant-in-Aid for Young Scientists
(B) No.\ 22710143, the Ministry of Education, Culture, Sports,
Science and Technology, and by Grant-in-Aid for Scientific Research (B) No.\  2271014, Japan Society for the Promotion of Science. 

\appendix

\bibliographystyle{abbrv}
\bibliographystyle{apalike}

\bibliographystyle{agsm}
\bibliography{dual_model_bib}

\end{document}